\documentclass[12pt]{article}

\usepackage[margin=1in]{geometry}
\usepackage[utf8]{inputenc}

\makeatletter
\renewcommand\thefootnote{\textsuperscript{\@fnsymbol\c@footnote}}
\let\old@thanks\thanks 
\DeclareRobustCommand\thanks[2][]{
  \AddToHook{begindocument/end}{
    \if\relax#1\relax%
      \footnotemark%
    \else%
      \protect\refstepcounter{footnote}\protect\label{#1}%
    \fi%
    \protected@xdef\@thanks{%
      \@thanks\protect\footnotetext[\the\c@footnote]{#2}%
    }%
  }%
}
\AddToHook{begindocument/begin}{\let\thanks\old@thanks} 
\let\old@maketitle\maketitle
\def\maketitle{\old@maketitle\def\thefootnote{\@arabic\c@footnote}}
\makeatother

\makeatletter
\def\@fnsymbol#1{\ensuremath{\ifcase#1\or 1 \or 2 \or
\mathsection\or \mathparagraph\or \|\or **\or \dagger\dagger
\or \ddagger\ddagger \else\@ctrerr\fi}}
\makeatother

\author{Sadashige Ishida\ref{a} and  Hugo Lavenant\ref{b}}
\thanks[a]{Email: sadashige.ishida@ist.ac.at \\ \hspace*{1.5em} Institute of Science and Technology Austria, Klosterneuburg, Austria}

\thanks[b]{Corresponding author\\ \hspace*{1.5em} 
Email: hugo.lavenant@unibocconi.it \\ \hspace*{1.5em} Department of Decision Sciences and BIDSA, Bocconi University, Milan, Italy}

\usepackage{amssymb,amsmath,amsfonts,amsthm,moresize,xfrac,nicefrac,accents}
\usepackage{cancel}
\usepackage{mathtools}
\usepackage{MnSymbol}
\usepackage{dsfont}
\usepackage{mathrsfs}
\usepackage{xcolor}

\usepackage{hyperref}
\hypersetup{
    colorlinks=true, 
    linkcolor=black,
    filecolor=magenta,      
    urlcolor=black,
    citecolor=blue,
}

\def\ie{\emph{i.e. }}

\let\on=\operatorname

\def\RR{\mathbb{R}}

\def\ZZ{\mathbb{Z}}

\def\cB{\mathcal{B}}
\def\cC{\mathcal{C}}

\def\cG{\mathcal{G}}

\def\cK{\mathcal{K}}
\def\cL{\mathcal{L}}
\def\cM{\mathcal{M}}

\def\cP{\mathcal{P}}

\def\cS{\mathcal{S}}

\def\indicator{I}
\def\characteristic{\mathds{1}}

\newcommand{\ddr}{\mathrm{d}}

\numberwithin{equation}{section}

\theoremstyle{plain} 
\newtheorem{theorem}{Theorem}[section]
\newtheorem{lemma}[theorem]{Lemma}
\newtheorem{proposition}[theorem]{Proposition}
\newtheorem{corollary}[theorem]{Corollary}

\theoremstyle{definition}
\newtheorem{definition}[theorem]{Definition}

\newtheorem*{assumptions}{Assumptions}

\theoremstyle{remark}
\newtheorem{remark}[theorem]{Remark}

\usepackage{tikz}
\usepackage{pgfplots}
\pgfplotsset{compat=1.13}

\newcommand{\COMMENT}[3]{}

\definecolor{revcolor}{rgb}{0.0, 0.5, 1.0}

\date{\today}

\begin{document}

\title{Quantitative convergence of a discretization of dynamic optimal transport using the dual formulation}

\maketitle

\begin{abstract}
We present a discretization of the dynamic optimal transport problem for which we can obtain the convergence rate for the value of the transport cost to its continuous value when the temporal and spatial stepsize vanish. This convergence result does not require any regularity assumption on the measures, though experiments suggest that the rate is not sharp. Via an analysis of the duality gap we also obtain the convergence rates for the gradient of the optimal potentials and the velocity field under mild regularity assumptions. To obtain such rates, we discretize the dual formulation of the dynamic optimal transport problem and use the mature literature related to the error due to discretizing the Hamilton-Jacobi equation. 

\medskip

\noindent \emph{Keywords}. Optimal transport; Hamilton-Jacobi equation; convex optimization.

\medskip

\noindent \emph{2020 Mathematics Subject Classification}. Primary 49Q22; Secondary 65K10, 49L12.
\end{abstract}

\section{Introduction}

\paragraph{The dynamic optimal transport problem and its discretization}

In this work, we are interested in the dynamic optimal transport problem. Given two probability measures $\mu, \nu$ over a spatial domain $\Omega$, it reads 
\begin{equation}
\label{eq:intro_problem}
\inf_{(\rho_t, v_t)} \int_0^1 \int_\Omega L(v_t) \, \ddr \rho_t \, \ddr t,
\end{equation}
where $L : \RR^d \to \RR$ is a given convex function and the infimum is taken over all pairs $(\rho_t, v_t)$ of a time-dependent probability distribution and velocity field which are solutions of
\begin{equation}
\label{eq:intro_constraint}
\begin{cases}
\partial_t \rho_t + \nabla \cdot (\rho_t v_t) = 0, \\
\rho_0 = \mu, \quad \rho_1 = \nu,
\end{cases}
\end{equation}
that is of the continuity equation with temporal boundary conditions $\mu$ and $\nu$. Here $\rho_t, v_t$ are indexed by a temporal variable $t \in [0,1]$ and $\nabla \cdot$ stands for the divergence with respect to the spatial variable. The interpretation is that $(\rho_t)$ must join $\mu$ and $\nu$ while being transported by the flow of $(v_t)$, at a minimal cost. Originally introduced by Benamou and Brenier in \cite{Benamou2000ACF} for numerical purposes as it is linked to the (static) optimal transport problem with cost $c(x,y) = L(y-x)$, this formulation turned out to be very fruitful. From a theoretical point of view, it is a robust formulation which enables to extend and generalize the optimal transport problem: it is used for optimal transport on graphs \cite{maas2011gradient}, for unbalanced optimal transport \cite{liero2018optimal, chizat2018unbalanced, kondratyev2016new}, for optimal transport of matrix-valued measures~\cite{brenier2020optimal} to cite a few extensions. From the numerical point of view, in addition to being one of the first methods proposed to solve the optimal transport problem in dimension more than one, it can be adapted to a great variety of related problems: Wasserstein gradient flow \cite{benamou2016augmented, carrillo2022primal}, mean field games \cite{benamou2015augmentedMFG, benamou2017variational}, and trajectory inference \cite{schmitzer2019dynamic} to mention a few. 

The usual road to solve~\eqref{eq:intro_problem} with the constraints~\eqref{eq:intro_constraint} is to first rewrite it as a convex problem by using the momentum $m_t = \rho_t v_t$ as an unknown rather than $v_t$. In this case both $m_t$ and $\rho_t$ are measures, and $v_t$ is recovered as the Radon-Nikodym derivative $\ddr m_t / \ddr \rho_t$  of $m_t$ with respect to $\rho_t$. This leads to
\begin{equation}
\label{eq:intro_problem_convex}
\inf_{(\rho_t, m_t)} \int_0^1 \int_\Omega L \left( \frac{\ddr m_t}{\ddr \rho_t} \right) \, \ddr \rho_t \, \ddr t, \quad \text{with constraints} \quad  
\begin{cases}
\partial_t \rho_t + \nabla \cdot m_t = 0, \\
\rho_0 = \mu, \quad \rho_1 = \nu.
\end{cases}
\end{equation}
The constraint $\partial_t \rho_t + \nabla \cdot m_t = 0$ is now linear, and the functional to be minimized is convex as $(x,y) \mapsto L(x/y) y$ is a convex function both in $x$ and $y$ when extended to $+ \infty$ for $y \leqslant 0$, except at $(0,0)$ where it is $0$. Then one proposes a finite dimensional version of the problem~\eqref{eq:intro_problem_convex} where both the time and space variables are discretized, and solves the resulting finite dimensional convex problem with standard methods in non-smooth convex optimization. We refer to \cite{papadakis2014optimal,Lavenant2018,carrillo2022primal,NataleTodeschi2021finitevolume,natale2022mixed} for instantiations of this approach with plain optimal transport. This comes with two challenges from the viewpoint of numerical analysis:
\begin{enumerate}
\item Guarantee (quantitatively) the convergence of the convex optimization solver used to solve the discretized problem. 
\item Guarantee (quantitatively) the convergence of the value and solutions to the discretized problem to the original one~\eqref{eq:intro_problem_convex} when the temporal and spatial stepsizes vanish.  
\end{enumerate}
The first point is a question of convex optimization which is quite well understood and that we will not address referring to \cite{papadakis2014optimal, hug2020convergence}. \emph{Our main concern is rather the second question}. The difficulty lies in the roughness of the functional to optimize: $(x,y) \mapsto L(x/y) y$ is only lower semi-continuous and discontinuous at $(0,0)$. Moreover a priori the data $\mu, \nu$ could be allowed to be arbitrary probability distributions (even Dirac masses), and there is no regularizing effect, so that $\rho_t$ could be a Dirac mass for every $t \in [0,1]$. Thus one really has to face the discontinuity of the functional to optimize. Nevertheless, numerical examples in the aforementioned works suggest that convergence holds even when the measures $\mu, \nu$ are quite rough.   

In \cite{lavenant2021unconditional} building on the previous works \cite{carrillo2022primal,gladbach2020scaling,erbar2020computation}, the second author gave an answer to the second question by providing sufficient conditions for any discretization of the dynamic optimal transport problem to indeed converge to the original problem when the temporal and spatial stepsizes vanish. This framework was later used in \cite{natale2022mixed} and extended to matrix-valued optimal transport in \cite{li2020general,li2023convergence}. However the arguments were inspired by the theory of $\Gamma$-convergence and were \emph{not} quantitative, that is, they did not come with a convergence rate.

The goal of this article is to provide \emph{quantitative} convergence rates of a discretized version of~\eqref{eq:intro_problem_convex} to the original problem, as the temporal and spatial stepsize vanish. There are few results already available: the works \cite{carrillo2022primal} and \cite{NataleTodeschi2021finitevolume} show that their respective methods reach first order of convergence. Specifically, for $h$ the (common) temporal and spatial stepsize, the error in the value of problem~\eqref{eq:intro_problem_convex}, that is, the transport cost, is of order $h$. The article \cite{NataleTodeschi2021finitevolume}  also extends this convergence to the solutions of the problem via an analysis of the dual gap. However, in both cases this rate is only available when the data $\mu, \nu$ have smooth densities which are bounded from below, and moreover the solution $\rho$ to the optimal transport problem needs to have a smooth density bounded from below. The latter assumption (positivity of the solution $\rho$ everywhere) is very strong, not directly implied by $\mu, \nu$ smooth and bounded from below \cite{santambrogio2016convexity}. Moreover, as we said above, in practice the discretizations behave well even when $\mu, \nu$ or $\rho$ are not bounded from below.

\begin{center}
\fbox{\begin{minipage}{0.9\textwidth}
In this work we propose a new discretization of the dynamic optimal transport problem (see Section~\ref{sec:discrete_OT}) for which we are able to quantify the convergence, at least of the value of the transport cost. With $h$ the (common) temporal and spatial stepsize the error is of order $\sqrt{h}$ but our result does not require \emph{any} assumption on the probability distributions $\mu, \nu$, besides that they have bounded support (see Theorem~\ref{th:convergence_of_cost}). 
\end{minipage}}
\end{center}

\noindent Let us already emphasize that this result comes with two important limitations. First we are not able to show that the rate improves if $\mu, \nu$ have a smooth density, and the numerical experiments we conducted suggest that our rate is not sharp. To prove better convergence rates for more regular inputs, we would need a refinement of the previous result \cite{CrandallLions_discreteHJ}, which we largely rely on (see Remark \ref{rmk:CL_determines_rate}).

The second point is about the efficiency of numerical computation. Our discretization is restricted to periodic boundary conditions although we can treat the transport between \(\mu,\nu\) on a bounded domain without any theoretical limitations by taking the entire periodic domain large enough. A full extension to bounded domains for more numerical efficiency may be potentially done, but is out of the scope of the present paper (see Remark~\ref{rmk:optimal_transport_domain}).

\paragraph{Intuition of our proposed discretization}

Our strategy to get such rates is the following. Rather than solving~\eqref{eq:intro_problem_convex}, we look at the dual problem which is known to be 
\begin{equation}
\label{eq:intro_problem_dual}
\sup_{\varphi} \int_\Omega \varphi(1,\cdot) \, \ddr \nu - \int_\Omega \varphi(0, \cdot) \, \ddr \mu \quad \text{with constraint} \quad   \partial_t \varphi + H(\nabla \varphi) \leqslant 0,
\end{equation}
where $\varphi = \varphi(t,x)$ for $t \in [0,1]$ and $x \in \Omega$ is a time-space dependent scalar function which is subsolution of the \emph{Hamilton-Jacobi} equation $\partial_t \varphi + H(\nabla \varphi) = 0$, being $H$ the Legendre transform of $L$.  

Discretizations of the Hamilton-Jacobi equation in a finite difference fashion have been studied in the context of viscosity solutions starting with the seminal work of Crandall and Lions \cite{CrandallLions_discreteHJ} and the rate of convergence is proved $\sqrt{h}$ under fairly general assumptions. This holds specifically for the initial value problem
\begin{equation*}
\begin{cases}
\partial_t \varphi + H(\nabla \varphi) = 0, \\
\varphi(0,\cdot) \text{ given},  
\end{cases}
\end{equation*}
for the Hamilton-Jacobi equation, importantly under an a priori Lipschitz bound on $\varphi(0, \cdot)$. In the context of optimal transport, an optimal solution $\varphi$ to the dual problem~\eqref{eq:intro_problem_dual} is actually known to be Lipschitz in space, with Lipschitz constant depending on the diameter of $\Omega$ but independent of $\mu, \nu$ (see Proposition~\ref{prop:bound_of_optimal_potential}). In our analysis, we also need to ensure that an optimal solution to the discretized version of~\eqref{eq:intro_problem_dual} is also Lipschitz. As it cannot be guaranteed a priori, we add Lipschitz continuity at $t=0$ as an additional constraint in our discretized problem; see Definition~\ref{def:discrete_dynamic_optimal_transport} and in particular the constraint~\eqref{eq:discrete_initial_Lipschitz}. 

\begin{center}
\fbox{\begin{minipage}{0.9\textwidth}
In short we propose a discretization of the dual problem~\eqref{eq:intro_problem_dual} where we use a finite difference discretization of the Hamilton-Jacobi equation for which convergence rates are already studied, adding a Lipschitz constraint to $\varphi(0,\cdot)$ in the discrete dual problem.      
\end{minipage}}
\end{center}

\paragraph{A brief comment on rates for other numerical methods}   

Dynamic optimal transport is not the only way to solve the optimal transport problem, we refer to \cite{peyreComputationalOT} and references therein for a comprehensive introduction. When faced with the linear programming formulation and its entropic regularization, the standard setting is to assume that measures are approximated via i.i.d. samples and rates should be understood in a statistical setting (that is, written in probability or in expectation as a function of the sample size) rather than a numerical analysis one (that is, written as a deterministic function of the stepsizes). We refer to \cite{weed2019sharp} and \cite{mena2019statistical} as well as references therein for results in this direction. In semi-discrete optimal transport, that is, when only one of the two measures is discrete, rates have been investigated both in a statistical setting \cite{altschuler2022asymptotics}, but also in a more standard setting of numerical analysis, related to the stability of the Monge-Ampère equation \cite{berman2021convergenceRates}. As far as PDEs methods are concerned, in addition to dynamic optimal transport, one could also solve the Monge-Ampère equation to compute a transport map \cite{benamou2014numerical}. The convergence when the resolution of discretization is reduced has been established in the context of viscosity solutions \cite{benamou2012viscosity}; see also \cite{benamou2019minimal,bonnet2022monotone}. Note that these convergence results concern a different method, so comparison with the present work is hard to do. We only point out that the results are not quantitative and typically need to assume some regularity of the input measures $\mu, \nu$ (like absolute continuity) in order to be able to write the Monge-Ampère equation in the first place.

\paragraph{Organization}

We first define properly the dynamic optimal transport problem, and we present the Hamilton-Jacobi equation together with its finite difference discretization: this is a well-understood theory that we summarize in Section~\ref{sec:preliminaries}. We move to our proposed discretization in Section~\ref{sec:discrete_OT}. Our main result, the quantitative convergence rate of the optimal transport cost, is presented in Section~\ref{sec:convergence_of_cost}. With a standard analysis of the duality gap, we prove, under an additional regularity assumption, that we obtain quantitative convergence of some variables (the velocity field $v$ and $\nabla \varphi$) from their discrete to their continuous counterpart in Section~\ref{sec:convergence_optimizers}. We numerically illustrate our results on simple one-dimensional test cases in Section~\ref{sec:numerics}: this shows that our rates are likely not sharp.  
\section{Settings and preliminaries}
\label{sec:preliminaries}

In this section, we present our setting and review the previous works that are going to be ingredients of our discretization of dynamic optimal transport.

\begin{assumptions}
We assume the following in the rest of the article.
\begin{itemize}

\item We restrict to $\Omega := \RR^d/ (D\ZZ^d) $ to be a $d$-dimensional torus with diameter 
\begin{equation*}
\operatorname{diam}(\Omega) = \frac{\sqrt{d}}{2} D.
\end{equation*}

\item The measures $\mu, \nu$ are Borel probability measures on $\Omega$ and we do not make any assumption such as absolute continuity unless otherwise stated. 

\item We take $L : \RR^d \to [0, + \infty)$, the ``Lagrangian'' which is a non-negative, strictly convex and superlinear function.
\end{itemize}
\end{assumptions}

\noindent As for the first point about the domain, we could assume $D=1$, or $\operatorname{diam}(\Omega)=1$ without loss of generality. However we prefer to keep it that way to emphasize how some constants depend on $\operatorname{diam}(\Omega)$. The last condition for the Lagrangian includes the most common choice $L(v) = |v|^p/p$ for $p \in (1, + \infty)$. 


\subsection{Dynamic (and static) optimal transport}

We briefly recall some ingredients of the dynamic optimal transport problem. We refer to the textbooks \cite{santambrogio2015optimal,peyreComputationalOT,villani2003topics} for additional details. 

We directly move to the convex formulation already mentioned in~\eqref{eq:intro_problem_convex} where we now define each term. The infimum will run over all the pairs $(\rho_t, m_t)_{t \in [0,1]}$  of probability distributions $\rho_t \in \mathcal{P}(\Omega)$ and vector-valued measure $m_t \in \mathcal{M}(\Omega)^d$ such that $t \mapsto \rho_t$ is continuous for the weak topology with $\rho_0 = \mu$ and $\rho_1 = \nu$. The continuity equation 
\begin{align*}
    &\partial_t \rho_t + \nabla \cdot m_t=0,
\end{align*}
is meant in the sense of distributions over the space $[0,1] \times \Omega$. For the functional to be optimized, we need a bit more of notations. The so-called Benamou-Brenier functional is defined as a functional for measures \(\rho \in \cM(\Omega)\) and \(m \in \cM^d(\Omega)\) by,
\begin{align*}
    \cB(\rho,m):=\sup_{(a,b)} \int_\Omega a \ \ddr \rho  + \int_\Omega b \cdot \ddr m 
\end{align*}
where the pair \((a,b)\) runs through \(C_b(\Omega; K)\) the space of continuous bounded functions valued in the convex domain \(K\) given by
\begin{align*}
    K:=\left\{(s,w) \in \RR \times \RR^d \text{ such that } s+H(w)\leqslant 0\right\}.
\end{align*}
Here the function \(H : \RR^d\rightarrow \RR\) is called the Hamiltonian, given as the Legendre transform of \(L\) by $H(w)=\sup_v \langle v,w \rangle - L(v)$. With a slight variation of \cite[Proposition 7.7]{santambrogio2015optimal}, it can be proved that if $\cB(\rho,m) < + \infty$ then $\rho$ is a non-negative measure and the measure $m$ is absolutely continuous with respect to $\rho$, and in this case
\begin{equation*}
\cB(\rho,m) = \int_\Omega L(v) \, \ddr \rho,
\end{equation*}
being $v \in L^1(\Omega,\rho)^d$ the Radon-Nikodym density of $m$ with respect to $\rho$.  With these notations the problem~ \eqref{eq:intro_problem_convex} now reads 
\begin{equation}
\label{eq:dynamic_formulation_OT}
\cK(\mu,\nu) = \inf_{(\rho_t,m_t)} \int_0^1 \cB(\rho_t,m_t) \, \ddr t \quad \text{such that} \quad  
\begin{cases}
\partial_t \rho_t + \nabla \cdot m_t = 0 \text{ weakly}, \\
\rho_0 = \mu, \quad \rho_1 = \nu.
\end{cases}
\end{equation} 
The existence of a minimizer to the problem can be shown by the direct method of calculus of variations (see e.g. \cite[Section 2]{chizat2018unbalanced} for a proof in a more general context), or alternatively by building it via an optimal solution for the static primal problem~\eqref{eq:static_ot} introduced below as in~\eqref{eq:competitor_dynamic_from_static}. We do not include the proof as it is not our main concern.  

\begin{theorem}[Existence of a solution in the primal problem]
Under our assumptions the infimum in~\eqref{eq:dynamic_formulation_OT} is attained. 
\end{theorem}

\noindent The question of uniqueness is more subtle and one would go to the static problem introduced below to analyze it. The outcome is that the minimum may not be unique, but it is so if at least one of the two measures $\mu, \nu$ is absolutely continuous with respect to the Lebesgue measure~\cite[Theorem 1.25]{santambrogio2015optimal}. 

Being a convex optimization problem under constraint, the dynamic optimal transport problem has a dual form. It reads
\begin{align}\label{eq:dual}
    \cK(\mu,\nu)=\sup_{\varphi} \int_\Omega \varphi(1,\cdot) \, \ddr \nu- \int_\Omega \varphi(0,\cdot) \, \ddr \mu \quad \text{such that} \quad \partial_t \varphi+H(\nabla \varphi)\leqslant 0 
\end{align}
where \(\varphi\) runs over Lipschitz functions defined over $[0,1] \times \Omega$, and the Hamilton-Jacobi constraint $\partial_t \varphi+H(\nabla \varphi)\leqslant 0$ means that $\varphi$ is a viscosity subsolution of $\partial_t \varphi+H(\nabla \varphi)= 0$ as we explain below in Section \ref{sec:viscosity_solutions}. Here $\varphi$ should be interpreted as a Lagrange multiplier for the continuity equation. The key result is that, not only a solution to the dual problem exists, but it has some Lipschitz regularity. 

As $L$ is convex, it is Lipschitz on bounded domains and we denote by $\operatorname{Lip}(L, B_r)$ its Lipschitz constant on the ball $B_r$ centered at $0$ and of radius $r$. On the other hand, $\mathrm{Lip}(\psi)$ for a real-valued function $\psi$ stands for its Lipschitz constant on its whole domain of defintion.   

\begin{proposition}[Existence of a solution in the dual problem]
\label{prop:bound_of_optimal_potential}
Under our assumptions there exists an optimal potential \(\overline{\varphi}\) in the dual dynamic transport problem satisfying
    \begin{align}
    \label{eq:bound_of_optimal_potential}
        \operatorname{Lip}(\overline{\varphi}(t,\cdot))  \leqslant \operatorname{Lip}(L, B_{\operatorname{diam}(\Omega)}), \quad \forall t \in [0,1].
    \end{align}
\end{proposition}

\noindent We delay the proof of this result until the next section as we need additional preliminary results, including the static formulation. 

The \emph{dynamic} formulation of optimal transport is to be contrasted with its \emph{static} one, which we introduce for the sake of completeness, and with which the reader may be more familiar. The cost function $c : \Omega \times \Omega \to [0, + \infty)$ associated to our Lagrangian $L$ is
\begin{equation*}
c(x,y) := \inf_k \{ L(y-x-k), \, k \in D \ZZ^d  \}
\end{equation*} 
as we are on the torus. The transport cost~\eqref{eq:dynamic_formulation_OT} is actually equal to:
\begin{align}\label{eq:static_ot}
    \cK(\mu,\nu):=\inf_{\pi\in \operatorname{ADM}(\mu,\nu)}\int_{\Omega\times \Omega}c(x,y) \, \ddr \pi(x,y)
\end{align}
where \(\operatorname{ADM}(\mu,\nu)\) is the set of the probability measures on $\Omega \times \Omega$ whose marginals are \(\mu\) and \(\nu\). Given a solution $\pi$ to the static optimal transport problem, one can construct a solution to the dynamic optimal transport problem as follows: for any $t \in [0,1]$, being $\Gamma_t : (x,y) \in \Omega \to (1-t)x + ty \in \Omega$ (where the addition is understood modulo $D \ZZ^d$), one sets for any Borel set $A$,
\begin{equation}
\label{eq:competitor_dynamic_from_static}
\rho_t(A) = \pi(\Gamma_t^{-1}(A)), \qquad m_t(A) = \int_{\Gamma_t^{-1}(A)} (y-x) \, \ddr \pi(x,y).   
\end{equation}
The intuition is that, once its initial and final position are chosen via the coupling $\pi$, each particle travels at constant speed on a straight line.
Finally, we introduce the dual formulation of the static optimal transport problem. It reads 
\begin{align}\label{eq:static_dual_cost}
    \cK(\mu,\nu)=\sup_{\psi} \int_\Omega \psi^c \, \ddr \nu +  \int_\Omega \psi \, \ddr \mu 
\end{align}
where  \(\psi^c\) is the \(c-\)transform of \(\psi\) given by $\psi^c(y)\coloneq  \inf_x c(x,y) - \psi(x)$ and \(\psi\) runs through  functions on \({\Omega}\) such that $\psi=\phi^c$ for some $\phi \in C(\Omega)$.
When a function \(\overline{\psi}\) attains the maximum, it is called a Kantorovich potential. The link between the static dual problem and the dynamic dual problem will be made clear in the next section. 

\subsection{Viscosity solutions to the Hamilton-Jacobi equation}
\label{sec:viscosity_solutions}

We briefly review the notion of viscosity solutions as it plays a central role in this article and highlights PDE aspects of optimal transport.  
The Hamilton-Jacobi equations often do not admit a classical solution, but naive notions of weak solution such as continuous functions differentiable almost Lebesgue everywhere can be too weak so that infinitely many solutions may exist. To ensure both existence and weakness, Evans \cite{Evans1980OnSC} and  Crandall and Lions \cite{CrandallLions_viscositySolution} independently introduced so-called viscosity solutions. For a Hamilton-Jacobi equation of the form
\begin{align*}
    \partial_t \varphi
     + H(\nabla \varphi)=0
\end{align*}
on a domain \([0,1]\times \Omega\), a Lipschitz function \(\varphi\) is said to be a viscosity subsolution (resp. supersolution) if, for any test function \(f\in C^1([0,1] \times \Omega)\), any local maximum \(x_0\) of \(f-\varphi\) (resp. local minimum) satisfies 
\begin{align*}
    & \partial_t f (x_0)+H\left(\nabla f(x_0)\right)\leqslant 0 \quad
\text{(resp. }  \partial_t f (x_0)+H\left(\nabla f(x_0)\right)\geqslant 0 \text{)}.
\end{align*} 
The inequality constraint of the dynamic dual optimal transport~\eqref{eq:dual} means that a competitor \(\varphi\) is required to be a viscosity subsolution. 

When \(\varphi\) is both a viscosity subsolution and a viscosity supersolution, it is called a viscosity solution. This definition indeed guarantees the existence of a unique solution. We present here the statement of \cite{CrandallEvansLions1984viscosity} summarizing the results of \cite{CrandallLions_viscositySolution}: even though it is phrased in $\RR^d$, it can be adapted at no cost on $\Omega = \RR^d / (D \ZZ^d)$ by simply identifying a function on $\Omega$ with a $D$-periodic function on $\RR^d$. 

\begin{theorem}[Existence and uniqueness of viscosity solution]\label{th:CandL_viscosity_solution}
    Let \(H : \RR^d\to \RR\) be continuous and \(\varphi_0 : \Omega \to \RR\) be Lipschitz. Then there exists exactly one viscosity solution to the initial value problem,
    \begin{equation*}
        \begin{cases}
        \partial_t \varphi  + H(\nabla \varphi)=0 & \textup{in } [0,\infty)\times \Omega,  \\
        \varphi(0,\cdot)=\varphi_0  & \textup{in }  \Omega.
        \end{cases}
    \end{equation*}
    Moreover, the Lipschitz constant does not increase in time \ie 
    \begin{align*}
        \operatorname{Lip}(\varphi(t',\cdot)) \leqslant \operatorname{Lip}(\varphi(t,\cdot)) \quad \textup{ for } t'>t.
    \end{align*}
\end{theorem}

\noindent Actually, in this simple setting where the Hamiltonian is not dependent on the spatial variable, the unique viscosity solution is characterized by the Hopf-Lax formula \cite[Theorem 3 in Chapter 10.3]{evans10PDE}: with the assumptions and notations of the theorem it is equal to
\begin{align}
\label{eq:Hopf-Lax}
\varphi (t,x) \coloneqq \inf_y \; \varphi_0(y)+tL\left(\frac{x-y}{t}\right).
\end{align}
The Hopf-Lax formula is a key ingredient to bridge the dual formulations of optimal transport in its static~\eqref{eq:static_dual_cost} and dynamic~\eqref{eq:dual} form. This is also what enables us to prove Proposition~\ref{prop:bound_of_optimal_potential}. 

\begin{proof}[\textbf{Proof of Proposition~\ref{prop:bound_of_optimal_potential}}]
We can choose an optimal Kantorovich potential \(\overline{\psi}\) for the static dual problem~\eqref{eq:static_dual_cost} so that it is Lipschitz continuous with \(\operatorname{Lip}(\overline{\psi}) \leqslant \sup_y \operatorname{Lip}(c(\cdot,y)) \leqslant \operatorname{Lip}(L, B_{\operatorname{diam}(\Omega)}) \)~\cite[Section 1.2.]{santambrogio2015optimal}. Then we define \(\overline{\varphi}\) as the unique viscosity solution of the Hamilton-Jacobi equation with initial data \(-\overline{\psi}\). Observe by~\eqref{eq:Hopf-Lax} that \(\overline{\varphi}(1,\cdot)=\overline{\psi}^c \) and hence 
\begin{align*}
    \int _\Omega \overline{\varphi}(1,\cdot) \, \ddr \nu - \int_\Omega \overline{\varphi}(0,\cdot) \, \ddr \mu= \int _\Omega \overline{\psi}^c \, \ddr \nu + \int_\Omega \overline{\psi}\, \ddr \mu = \cK(\mu,\nu).
\end{align*}
Therefore, \(\overline{\varphi}\) is not merely an admissible competitor but is an optimal potential for the dynamic dual problem~\eqref{eq:dual}. Finally Theorem \ref{th:CandL_viscosity_solution} asserts that \(\operatorname{Lip}(\varphi(t,\cdot)) \leqslant \operatorname{Lip}(\varphi(0,\cdot)) = \operatorname{Lip}(\overline{\psi}) \) for \(t\in[0,1]\).
\end{proof}

\subsection{Discretization of the Hamilton-Jacobi equation}\label{sec:discrete_HJ}

As we mentioned in the introduction, our key idea is to discretize not the primal formulation~\eqref{eq:dynamic_formulation_OT} but rather the dual formulation~\eqref{eq:dual}. Indeed discretization of the Hamilton-Jacobi equations is a widely studied topic, in particular since the seminal work of Crandall and Lions \cite{CrandallLions_discreteHJ}. 

We adapt Crandall and Lions's original setting of the domain \([0,\infty) \times \RR^d\) for \(Q=[0,1]\times (\RR^d / (D \ZZ^d))\). Their results are still valid on \(Q\) without additional conditions as we can simply extend functions on \(\Omega = \RR^d / (D \ZZ^d)\) by copying it infinite times for \(\RR^d\). 

The time domain \([0,1]\) is discretized by uniformly sampling points: we define \(T_D \coloneq \{0, \Delta t, \ldots, (N_T-1)\Delta t,1\} \) with some positive integer \(N_T\) and \(\Delta t =1/N_T\). The space domain \(\Omega=\RR^d / (D \ZZ^d)\) is discretized in the same way, that is we write \(\Omega_D \coloneq \{(j_1 \Delta x, \ldots, j_d \Delta x)\}\) 
with indices \((j_1,\ldots, j_d)\in (\ZZ/N_X \ZZ)^d \coloneqq \{0,1,\ldots,N_X-1\}^d\) with some positive integer \(N_X\) and \(\Delta x=D/N_X\). 
In the sequel, we write \(j\coloneqq(j_1,\ldots, j_d)\) and \(j\Delta x\coloneq (j_1 \Delta x, \ldots, j_d \Delta x)\) for simplicity.
Hence the discretization of \(Q\) is given by \(Q_D\coloneqq T_D \times \Omega_D\). 

Each continuous function \(\varphi\) is approximated by a discrete function \(\Phi : Q_D\to \RR\) given as the evaluation of \(\varphi\) at the grid points. For the value at \((i\Delta t,j\Delta x)\), we write \(\Phi^i_j\) or \(\Phi^{i}_{j_1,\ldots,j_d}\), but we will  occasionally omit subscripts \(i\) and \(j\) when we mean the collection \(\{\Phi^i_j\}^i_j\) or we do not need to specify \(i\) or \(j\). We denote by \(\RR^Z\) the collection of functions from a set \(Z\subseteq Q_D\) to \(\RR\), and define
\begin{align*}
    \| \Psi \|_{L^\infty(Z)}\coloneqq \max_{z\in Z} |\Psi(z)|.
\end{align*}
for function \(\Psi\in \RR^Z\).

Discretization of the Hamilton-Jacobi equation of the form \(\partial_t \varphi+H(\nabla \varphi)=0\) is given as follows. The initial state of the discrete function \(\Phi^0\) is given on the grid points \(\{0\}\times \Omega_D\subset Q_D\) and it is time-updated as
\begin{align*}
    \Phi^{i+1} =\cS( \Phi^{i} ),
\end{align*}
by a map \(\cS : \RR^{\Omega_D} \to \RR^{\Omega_D}\) called a {\em scheme}. We will deal with a certain class of schemes following a standard finite difference setting as in \cite{CrandallLions_discreteHJ}. A scheme \(\cS\) is said of {\em difference form} if there exists a function $\cG\colon \RR^{2d}\to \RR$ such that 
\begin{align*}
    \cS_j( \Psi ) = \Psi_{j} - \Delta t \cG\left(\frac{ {\Delta_{-,j}}\Psi}{\Delta x}, \frac{ {\Delta_{+,j}}\Psi}{\Delta x}\right),
\end{align*}
for each \(j\) where \({\Delta_{-,j}}\coloneqq(\Delta_{-,j}^1,\ldots,\Delta_{-,j}^d)\) takes the spatially backward difference
\begin{align*}
   {\Delta_{-,j}^k}\Psi \coloneqq \Psi_{j_1,\ldots,  j_{k}, \ldots, j_d }-  \Psi_{j_1,\ldots,  j_{k}-1, \ldots, j_d },
\end{align*}
for each $k$ and \({\Delta_{+,j}}\coloneqq(\Delta_+^1,\ldots,\Delta_+^d)\) takes the forward difference 
\begin{align*}
    {\Delta_{+,j}^k}\Psi \coloneqq \Psi_{j_1,\ldots,  j_k+1, \ldots, j_d }-  \Psi_{j_1,\ldots,  j_k, \ldots, j_d },
\end{align*}
for each \(k=\{1,\ldots,d\}\). Of course the expression $j_k \pm 1$ is understood modulo $N_X$ as we have periodic boundary conditions. We will omit the subscript \(j\) for \(\Delta_{\pm}\) when there is no confusion. The two important properties that a scheme can have are the following:
\begin{itemize}
\item \emph{Consistency}. A scheme \(\cS\) of  difference form is {\em consistent} with \(H\) if \(\cG\) satisfies $\cG(a,a)=H(a)$ for any \(a\in \RR^d\). It is equivalent to say that \(\cS\) gives the exact solution on grid points for any time-space affine function.
\item \emph{Monotonicity}. Let us define a subset of discrete space functions
 \begin{align*}
    \cC_R \coloneqq \left\{ \Psi\in \RR^{\Omega_D} \text{ such that } \forall k,j \,  \left| \frac{\Delta^k_{+,j} \Psi}{\Delta x} \right| \leqslant R \right\}
\end{align*}
for a fixed \(R>0\). The space $\cC_R$ can be interpreted as the discrete counterpart of the functions which are $R$-Lipschitz in each coordinate. A scheme is {\em monotone} on \([-R,R]\) if for each grid point $j$ the restriction of $\cS_j$ to $\cC_R$ is non-decreasing with respect to any of its variables. 
\end{itemize}

The striking result is that these two properties guarantee not only convergence of the discrete solutions to the continuous ones, but also a quantitative rate of convergence. Crandall and Lions showed a quantitative estimate for the approximation error of such discrete solutions on a family of time-space discretizations of domain \(Q\) \cite[Theorem 1]{CrandallLions_discreteHJ}. In their work, the ratio \(\zeta\coloneqq \Delta t/ \Delta x\)  within a family of discretizations is fixed. Throughout this article, we follow their setting, and for each discretization \(Q_D\) we denote its resolution by \(h\coloneqq \Delta t\).

\begin{theorem}[Convergence of discrete Hamilton-Jacobi equation]\label{th:CandL_discrete_bound}
    Let \(H : \RR^d\rightarrow \RR\) be continuous and  \(\varphi\) be the unique viscosity solution to the initial value problem,
    \begin{align*}
		\begin{cases}        
        \partial_t \varphi+H(\nabla \varphi)=0 \quad \textup{ in } Q=[0,1]\times \Omega,\\
        \varphi(0,\cdot)=\varphi_0,
         \end{cases}
    \end{align*}
    with initial data \(\varphi_0\) Lipschitz in each coordinate with Lipschitz constant \(R>0\).  
    Let \(Q_D=T_D \times \Omega_D\) be a discretization of \(Q\) with resolution \(\Delta t = h\) and ratio \(\zeta = \Delta t/ \Delta x\) fixed. Define a discrete solution \(\Phi\in \RR^{Q_D}\) by 
    \begin{align*}
    	\begin{cases}
        \Phi^{i+1}=\cS(\Phi^i), & i\in\{0,\ldots,N_T-1\},\\
        \Phi^0_j\coloneqq \varphi_0(j \Delta x), & \forall j,
		\end{cases}
    \end{align*}
    with a scheme \(\cS\) of  difference form which is consistent with \(H\) and monotone on \([-R-\delta,R+\delta]\) for some $\delta > 0$. 
    
Then we have the estimate,
    \begin{align*}
        \| \Phi - \varphi\|_{L^\infty(Q_D)}\leqslant C \sqrt{h}
    \end{align*} 
     with a constant \(C\) depending only on the scheme $\cS$, \(\| \varphi_0\|_{L^\infty(Q)}, R+\delta\), and \(H\). 
\end{theorem} 

\noindent This theorem will be the main result we will use to get our quantitative convergence rates for the optimal transport problem.

\begin{remark}
The theorem is usually stated with $\delta = 1$, that is, the scheme should be monotone on \([-R-1,R+1]\), but a close inspection of the proof in~\cite{CrandallLions_discreteHJ} reveals that any $\delta > 0$ is possible. 
\end{remark}

\begin{remark}[Lipschitz in each coordinate and why we restrict to a periodic setting]
\label{rmk:proof_regularizing_HJ}
We have done another modification compared to the classical statement of the theorem. We assume that $\varphi_0$ is $R$-Lipschitz in each coordinate instead of simply $R$-Lipschitz, that is, for every $k$ and $(x_1, \ldots x_{k-1}, x_{k+1}, \ldots,x_d) \in \RR^{d-1} / (D \ZZ^{d-1})$ the function 
\begin{equation*}
x \mapsto \varphi_0(x_1, \ldots x_{k-1}, x ,x_{k+1}, \ldots,x_d)
\end{equation*}
is $R$-Lipschitz. A $R$-Lipschitz function is $R$-Lipschitz in each coordinate, and a function which is $R$-Lipschitz in each coordinate is $\sqrt{d} R$-Lipschitz in the classical sense. Functions which are $R$-Lipschitz in each coordinate are the perfect analogue of the space $\cC_R$ of discrete functions, and this will be useful in the proof of our Theorem~\ref{th:convergence_of_cost}. 

An important step in the proof of the theorem is the propagation of the regularity for the discrete solution: one can prove that if $\Phi^0 \in \cC_R$ then 
\begin{align}
\label{eq:regularizing_effect_HJ}
\Phi^i \in \cC_R, \quad \forall i. 
\end{align}
It is obtained by combining two simple arguments. The first one is that, as the scheme commutes with the addition with constant functions and is monotone, it is non-expansive in $L^\infty$:
\begin{align*}
\|\cS(\Psi) - \cS(\Psi')\|_{L^\infty (\Omega_D)}  \leqslant  \|\Psi - \Psi'\|_{L^\infty (\Omega_D)} 
\end{align*}
for any pair of \(\Psi\) and \(\Psi'\) in \(\cC_R\)~\cite[Proposition 3.1]{CrandallLions_discreteHJ}. The second one is to apply the non-expansiveness to $\Psi'$, a shifted version of $\Psi$. As a shift in space commutes with the discrete differential operators $\Delta^k_+, \Delta^k_-$ for fixed $k$, we obtain 
\begin{align*}
\|\Delta_+^k\cS(\Psi)\|_{L^\infty(\Omega_D)}\leqslant \|\Delta_+^k\Psi\|_{L^\infty(\Omega_D)},
\end{align*}
and from there an immediate induction gives~\eqref{eq:regularizing_effect_HJ}. Note that the same argument at the continuous level gives that, if $\varphi_0$ is $R$-Lipschitz in each coordinate, then so is $\varphi(t,\cdot)$ for any $t \geqslant 0$.

Importantly, this regularizing effect of the scheme does not work if the spatial domain is no longer the torus nor the whole Euclidean space: the first estimate stays valid but the second argument about the commutativity breaks down on the boundary. For this reason, our work is phrased on a periodic domain rather than on a bounded domain: we will not rely on the estimate~\eqref{eq:regularizing_effect_HJ} in itself, but this estimate is necessary in order for Theorem~\ref{th:CandL_discrete_bound} to be true. 
\end{remark}

\subsubsection{Vanishing viscosity scheme}\label{sec:vanishing_viscosity_scheme}

An example of monotone and consistent schemes presented in \cite{CrandallLions_discreteHJ} is the so-called vanishing viscosity scheme, which is a discrete analogue of the vanishing viscosity method for the continuous Hamilton-Jacobi equation. For a discrete space function \(\Psi\in \RR^{\Omega_D}\), it is given by
\begin{align}
\label{eq:vanishing_viscosity_scheme}
    \cS(\Psi)=\Psi-\Delta t \left\{H(\nabla_{D} \Psi) - \varepsilon  \Delta_D \Psi  \right\},
\end{align}
with some \(\varepsilon>0\).
Here \(\nabla_D, \Delta_D\) are respectively the discrete centered gradient and the discrete Laplacian given by,
\begin{align*}
    \nabla_{D,j}\Psi
    \coloneqq \frac{ \Delta_{-,j}\Psi+ {\Delta_{+,j}\Psi}}{2\Delta x}, \qquad \Delta_{D,j}\Psi\coloneqq
    \sum_k \frac{{\Delta_{+,j}^k}\Psi-{\Delta_{-,j}^k}\Psi}{(\Delta x)^2}
\end{align*}
at each location \(j\).
With the (forward) discrete time derivative
\begin{align*}
    \partial_{\Delta t} ^i\Phi \coloneqq\frac{\Phi^{i+1}-\Phi^{i}}{\Delta t},
\end{align*}
the scheme \eqref{eq:vanishing_viscosity_scheme} reads
\begin{align*}
    \partial_{\Delta t} \Phi + H(\nabla_D \Phi) - \varepsilon  \Delta_D \Phi = 0,
\end{align*}
for a discrete time-space function \(\Phi\in \RR^{Q_D}\),
which is a discrete analogue of the Hamilton-Jacobi equation with the viscosity term. 

The consistency is immediate: for \(\Psi\in \RR^{\Omega_D}\) satisfying \(\Delta_{+} \Psi / \Delta x=\Delta_{-}\Psi / \Delta x =(a_1,\ldots,a_d)\), we have \(H(\nabla_{D}\Psi)=H(a_1,\ldots,a_d)\) and \( \Delta_{D} \Psi =0 \).
The monotonicity is not attained only with the Hamiltonian term, but can be ensured together with the viscosity term. 
A simple computation reveals that the scheme is monotone on $\cC_R$ when  \(\varepsilon\) satisfies 
\begin{align}
\label{eq:monotonicity_condition}
\frac{\operatorname{Lip}(H,B_R)}{2} \leqslant \frac{\varepsilon}{\Delta x} \leqslant \frac{\Delta x}{2d\Delta t} 
\end{align}
where $\operatorname{Lip}(H,B_R)$ is the Lipschitz constant of $H$ on the centered ball of radius $R$. Recalling that we assume that the ratio \(\zeta=\Delta t/\Delta x\) is fixed among a family of discretizations of \(Q\), we need to choose $\zeta \geqslant d \operatorname{Lip}(H,B_R)$, and in this case $\varepsilon$ will go to zero at rate $h = \Delta t \propto \Delta x$: the viscosity $\varepsilon$ vanishes together with the stepsize.

Several alternatives can be found in the literature, such as the upwind scheme also in \cite{CrandallLions_discreteHJ} (for $d=1$), higher order finite difference schemes~\cite{osher1991high}, and also discontinuous Galerkin methods~\cite{cheng2007discontinuous}. We, however, focus on the vanishing viscosity scheme in this article because of the following reasons: it is simple; we are able to solve efficiently the discrete dynamic optimal transport we build on it; and the function $\overline{\varphi}$ is a priori not expected to be better than Lipschitz uniformly over the space-time domain so using high order schemes seems less helpful as far as the theoretical analysis is concerned.

\section{Discrete dynamic optimal transport}
\label{sec:discrete_OT}

In this section, we introduce a discrete formulation of dynamic optimal transport. We first do so using a general monotone and consistent scheme, and then specifically with the vanishing viscosity scheme explained in the previous section.
\subsection{Discrete problem for a general scheme}\label{sec:discrete_OT_general_scheme}
We  discretize the dynamic dual problem~\eqref{eq:dual} on \(Q= [0,1] \times \Omega\). For the discretization of the domain, we use \(Q_D= T_D \times \Omega_D\) defined in Section \ref{sec:discrete_HJ}. 

Regarding the discretization of probability measures, we follow a standard approach. For a given \(\mu \in \cP(\Omega)\), we define its discretization \(\Pi\mu\) in a way that \(\Pi\mu \overset{\ast}\rightharpoonup \mu\)  as \(\Delta x \rightarrow 0\). A simple choice is the projection onto the Dirac measures on grid points, given by
\begin{align}\label{eq:discrete_probability_measure}
    \Pi\mu
    = \sum_j \mu({B_{j\Delta x}})\delta_{j\Delta x},
\end{align}
where \(B_{j\Delta x}\) is the \(d-\)dimensional half-open box centered at \(j\Delta x\) with edge length \(\Delta x\) given by,
  \begin{align*}
    B_{j\Delta x}=\big[(j_1-1/2) \Delta x, (j_1+1/2)\Delta x\big) \times \cdots \times \big[(j_d-1/2) \Delta x, (j_d+1/2)\Delta x\big).
  \end{align*} 
  
\begin{remark}[Choice of discrete measure]\label{rmk:asymmetry_mu_nu}
    In this article, we stick to this specific discretization of measures, but this is not the unique choice. For example, piecewise uniform measures in the boxes \(\{B_{j\Delta x}\}_j\) is also a reasonable option. All our results easily extend to this discretization of measures as soon as Lemma~\ref{lem:continuous discrete integral gap} is valid. 
\end{remark}

\paragraph{Clamped discrete gradient} With the above settings, we are almost ready to introduce our discrete optimal transport problem. We finally impose a constraint on discrete functions as follows. 
We will later guarantee the convergence of the discrete transport cost to the continuous one. For this result, we will need Theorem~\ref{th:CandL_discrete_bound} which requires a monotone scheme on \([-R-\delta,R+\delta]\) with $\delta > 0$ and an initial condition \(\Phi^0\in \cC_R\).
However, the boundedness of \(\Delta_{+}  \Phi^0 / \Delta x\) is not a priori guaranteed in our upcoming formulation of discrete optimal transport, in contrast to that, the gradient of an optimal potential can be chosen to be bounded by \( \operatorname{Lip}(L,B_{\operatorname{diam}(\Omega)}) \) in the continuous setting stated as Proposition~\ref{prop:bound_of_optimal_potential}. To cope with this problem, we explicitly constrain
\begin{align*}
\forall j,k, \quad    \left| \frac{ \Delta^k_{+,j} \Phi^0}{\Delta x} \right| \leqslant R
\end{align*}
for a parameter $R \geq \operatorname{Lip}(L,B_{\operatorname{diam}(\Omega)})$. Said differently we impose $\Phi^0 \in \cC_{\operatorname{Lip}(L,B_{R})}$. 

\begin{definition}[Discrete dynamic optimal transport]
    \label{def:discrete_dynamic_optimal_transport}
    Let \(\mu,\nu\in \cP(\Omega)\) and choose a parameter \( R \geq \operatorname{Lip}(L,B_{\operatorname{diam}(\Omega)})\). Let us assume a time-space discretization as defined so far, and \(\cS\) be a scheme of difference form that is monotone on \([-R-\delta,R+\delta]\) for some $\delta > 0$ and
 consistent with the Hamiltonian \(H\) of the continuous optimal transport problem.  We say that \emph{the discrete optimal transport cost} between \(\mu\) and \(\nu\) is
\begin{align}\label{eq:discrete_dual_cost}
    \cK_D(\mu,\nu)\coloneq \max_{\Phi}
     \int_\Omega \Phi^{N_T} \, \ddr \Pi\nu - \int_\Omega  \Phi^{0} \, \ddr \Pi\mu
\end{align}
where \(\Phi \in \RR^{Q_D}\) runs over the discrete functions satisfying,
\begin{align}
    &\Phi^{i+1} -\cS(\Phi^i ) 
    \leqslant 0,\label{eq:discrete_HJ_inequality}
    & \forall i \in \{0,\ldots,N_T-1\}, \\
    &\left|\frac{\Delta^k_{+,j} \Phi^0}{\Delta x}\right|\leqslant R & \forall j \in \{ 0, \ldots, N_X-1 \}^d, \, \forall k \in \{ 1, \ldots, d \}. \label{eq:discrete_initial_Lipschitz}
\end{align}
If a discrete function \(\overline{\Phi}\) attains \(\cK_D(\mu,\nu)\), we call \(\overline{\Phi}\) \emph{an optimal potential} for the discrete problem. 
\end{definition}

\noindent This definition clearly mimics~\eqref{eq:dual} which is the continuous dual problem. At this level of generality, this is not necessarily a concave maximization problem. It would depend on the precise choice of the scheme $\cS$. We emphasize that our convergence result for the transport cost (Theorem \ref{th:convergence_of_cost}) holds even in the absence of concavity.

\begin{remark}[Break of the symmetry between $\mu$ and $\nu$]
This discrete optimal transport is in general not symmetric between $\mu$ and $\nu$, that is $\cK_D(\mu,\nu) \neq \cK_D(\nu,\mu)$ as can be seen both from the Hamilton-Jacobi constraint~\eqref{eq:discrete_HJ_inequality} and the constraint~\eqref{eq:discrete_initial_Lipschitz} on the initial potential. This is in contrast with other discretizations like~\cite{papadakis2014optimal,Lavenant2018} which are symmetric in $\mu$ and $\nu$. 
\end{remark}

\begin{remark}\label{rmk:largerR}
The threshold $R$ for clamping in Definition \ref{def:discrete_dynamic_optimal_transport} is allowed to be greather than $\on{Lip}(L,B_{\on{diam}(\Omega)})$ as far as the scheme $\cS$ is monotone. But choosing a larger $R$ makes room for the scheme to be monotone smaller (see~\eqref{eq:monotonicity_condition}), and makes the convergence a bit slower in the sense that the multiplicative constant in front of $\sqrt{h}$ in Theorem~\ref{th:convergence_of_cost} increases with $R$. Thus choosing a larger $R$ would make sense only if one does not have an exact access to $\on{Lip}(L,B_{\on{diam}(\Omega)})$. 
\end{remark}

\subsection{Discrete problem with vanishing viscosity}\label{sec:discrete_OT_vanishing_viscosity}

We have introduced a discretization of dynamic optimal transport. Notice it is a discrete counterpart of the dynamic \emph{dual} formulation. In order to obtain quantities such as optimal measures and optimal velocities, we need a \emph{primal} formulation as well. This will not be needed for our main convergence result about the transport cost (Theorem~\ref{th:convergence_of_cost}), but will be necessary for the convergence of optimizers (Theorem~\ref{th:convergence_of_optimizer}) to make sense. 

We focus on the vanishing viscosity scheme explained in Section \ref{eq:vanishing_viscosity_scheme} as this is a simple and practical example of schemes that make the discrete optimal transport a convex problem. With this scheme the constraint \eqref{eq:discrete_HJ_inequality} is written as
\begin{align*}
    \partial_{\Delta t}\Phi+H(\nabla_D \Phi) - \varepsilon \Delta_D\Phi\leqslant 0.
\end{align*}
which is a convex constraint.

\begin{remark}(Range of admissible parameters)
\label{rmk:rangle_admissible_parameters}
Let's focus on the case where $L$ (and thus $H$) are radial. Note that the scheme must be monotone on $[-R-\delta, R+\delta]$ with $\delta > 0$ and $R \coloneqq \operatorname{Lip}(L,B_{\operatorname{diam}(\Omega)})$. But as $\nabla L$ and $\nabla H$ are the inverse to each other, the constraint~\eqref{eq:monotonicity_condition} reads
\begin{equation*}
\frac{\operatorname{diam}(\Omega)}{2} < \frac{\varepsilon}{\Delta x} \leqslant \frac{\Delta x}{2 d\Delta t}.  
\end{equation*}
Interestingly it does \emph{not} depend on $L$. 
\end{remark} 

To obtain the primal formulation, we begin with writing the dual problem (Definition \ref{def:discrete_dynamic_optimal_transport}) more concretely for the vanishing viscosity scheme. To make the exposition simple, we introduce some notations. We fix \(\mu,\nu\in\cP(\Omega)\) in the rest of the section and define a functional 
\(F_D :  \RR^{Q_D}\to \RR\) by
\begin{align}
\label{eq:definition_FD}
    F_D \Phi = \int_\Omega \Phi^{N_T} \, \ddr \Pi\nu- \int_\Omega \Phi^0 \, \ddr \Pi\mu,
\end{align}
and set a constant \(R \geq  \operatorname{Lip}(L,B_{\operatorname{diam}(\Omega)})\). Next, we concatenate discrete differential operators as follows. We define \(T_D'\coloneqq\{0,\Delta t,\ldots, (N_T-1)\Delta t \}\) by dropping the last time step \(N_T=1\) from \(T_D\)  and define a subset of the entire discrete space \(Q_D\) by \(Q_D'\coloneqq  T_D' \times \Omega_D\). We then define an operator \(A=(A_t,A_x,A_R) : \RR^{Q_D}\to \RR^{Q'_D +d\times Q'_D+ d\times\Omega_D}\) by
\begin{align*}
    &A_t \Phi=(\partial_{\Delta t}^0\Phi-\varepsilon \Delta_D \Phi^0, \ldots, \partial_{\Delta t}^{N_T-1}\Phi- \varepsilon \Delta_D \Phi^{N_T-1})\in \RR^{Q_D'},\\
    &A_x\Phi=(\nabla_D \Phi^0, \ldots, \nabla_D \Phi^{N_T-1}) \in \RR^{d\times Q_D'},\\
    &A_R \Phi=(\Delta_{+}\Phi^0/\Delta x ) \in \RR^{d\times\Omega_D}.
\end{align*}
We rewrite the problem with these settings.
\begin{definition}[Dual problem with vanishing viscosity]\label{def:dual_problem_vanishing_viscosity}
    Let \(\mu,\nu\in \cP(\Omega)\), \(H\) be the Hamiltonian of the continuous problem, and \(\cS\) be the vanishing viscosity scheme. The dual formulation of discrete transport problem is defined as
    \begin{align}\label{eq:dual_cost_vanishing_viscosity}
        & \sup_\Phi F_D\Phi,
    \end{align}
    with constraints
    \begin{align}
        & A_t\Phi+H(A_x \Phi)\leqslant 0,  \label{eq:discrete_HJ_vanishing_viscosity}\\
        & |A_R \Phi| \leqslant R. \label{eq:forward_difference_bound_vanishing_viscosity}
    \end{align}
\end{definition}

To derive the expression of the dual of this problem, which we call the \emph{primal} problem, we use the method of Lagrange multipliers. We introduce a new variable $\Sigma$ in the dual which at optimality coincides with $A \Phi$, and take a Lagrange multiplier $\Lambda$ to enforce the constraint $A \Phi = \Sigma$. Specifically the problem reads as the saddle point problem
\begin{align}
\label{eq:saddle_point_vanishing_viscosity}
   \cK_D(\mu,\nu) = \sup_{\Phi, \Sigma} \inf_{\Lambda} \; \cL(\Phi, \Sigma, \Lambda),
\end{align}
with the functional 
\begin{align*}
    \cL(\Phi, \Sigma, \Lambda)= F_D \Phi - \indicator_{\Sigma_t+H(\Sigma_x)\leqslant 0} - \indicator_{|\Sigma_R|\leqslant R} - \Lambda \cdot (A\Phi-\Sigma).
\end{align*}
The newly introduced symbols are defined as follows. 
We set the variable  \(\Sigma\coloneqq (\Sigma_t, \Sigma_x, \Sigma_R)\in\RR^{Q'_D +d\times Q'_D+ d\times\Omega_D}\). Each indicator function \(\indicator\) takes \(0\) if the constraint is satisfied, and otherwise \(+ \infty\).  The variables \(\Lambda\coloneqq(\Lambda_\rho, \Lambda_m, \Lambda_\eta)\in \RR^{Q'_D +d\times Q'_D+ d\times\Omega_D} \) are the Lagrange multipliers. Finally by ``$\cdot$'' we denote the standard Euclidean product between vectors in $\RR^N$, with a dimension $N$ which should be clear from context. It should be interpreted as the integral of a discrete scalar or vector field by a discrete scalar or vector-valued measure respectively, which is just the summation of element-wise products. 

Now let us consider the formal exchange of the infimum and the supremum. 
By a direct computation, we have
\begin{align*}\label{eq:Lagrangian_computation}
\inf_{\Lambda} \sup_{\Phi,\Sigma} & \; \cL(\Phi, \Sigma, \Lambda) = \inf_{\Lambda} \sup_{\Phi,\Sigma} \; (F_D-A^\top\Lambda) \cdot \Phi    - \indicator_{\Sigma_t+H(\Sigma_x)\leqslant 0} - \indicator_{|\Sigma_R|\leqslant R} +  \Lambda \cdot \Sigma \nonumber\\
    &=\inf_\Lambda \sup_{\Phi, \Sigma_t, \Sigma_x} \; (F_D-A^\top\Lambda) \cdot \Phi - \indicator_{\Sigma_t+H(\Sigma_x)\leqslant 0} + \Lambda_\rho \cdot \Sigma_t + \Lambda_m \cdot \Sigma_x + R \cdot |\Lambda_\eta|
\end{align*}
where in the last line we have performed the maximization over $\Sigma_R$ by taking $\Sigma_R = R \Lambda_\eta / |\Lambda_\eta|$ elementwise. If $\Lambda_\rho$ is not element-wise non-negative, then the supremum in $\Sigma_t \to - \infty$ would yield $+ \infty$. But once we know $\Lambda_\rho \geqslant 0$, we see that the supremum in \(\Sigma\) is attained in the boundary of the convex constraints, namely \(\Sigma_t=-H(\Sigma_x)\).  When we take the supremum in $\Sigma_x$ we see that $\Lambda_\rho$ should be strictly positive if $\Lambda_m$ is non-zero and that in this case 
\begin{equation*}
\sup_{\Sigma_x} \;  \Lambda_m \cdot \Sigma_x- \Lambda_\rho \cdot H(\Sigma_x) =   L\left(\frac{\Lambda_m}{\Lambda_\rho}\right) \cdot \Lambda_\rho
\end{equation*} 
Note that in this case we also have at optimality 
\begin{equation}
\label{eq:link_v_phi_optimality_discrete}
\frac{\Lambda_m}{\Lambda_\rho} =  \nabla H (\Sigma_x)=\nabla H (\nabla_D \Phi)   
\end{equation}
for element-wise non-zero \(\Lambda_\rho\). Thus the problem boils down to 
\begin{equation*}
\inf_{\Lambda} \sup_{\Phi,\Sigma} \; \cL(\Phi, \Sigma, \Lambda) = \inf_\Lambda  L\left(\frac{\Lambda_m}{\Lambda_\rho}\right) \cdot \Lambda_\rho + R \cdot |\Lambda_\eta|  +  \sup_{\Phi} \; \left\{ (F_D-A^\top\Lambda) \cdot \Phi \right\}
\end{equation*}
Writing $R \cdot |\Lambda_\eta| = R \|\Lambda_\eta\|_{L^1(\Omega_D)}$ and interpreting $\Phi$ as a Lagrange multiplier we obtain the following minimization problem.  

\begin{definition}[Primal problem with vanishing viscosity]\label{def:primal_problem_vanishing_viscosity}
    Let \(\mu,\nu\in \cP(\Omega)\), \(L\) be the Lagrangian of the continuous problem, and \(\cS\) be the vanishing viscosity scheme. The primal formulation of discrete transport problem is defined as
    \begin{align} \label{eq:primal_cost_vanishing_viscosity}
      \inf_\Lambda \;   L\left(\frac{\Lambda_m}{\Lambda_\rho}\right) \cdot  \Lambda_\rho +R \|\Lambda_\eta\|_{L^1(\Omega_D)}
    \end{align}
    where \(\Lambda\) runs through discrete functions satisfying the discrete continuity equation,
    \begin{align}\label{eq:continuity_equation_vanishing_viscosity}
        A^\top\Lambda=F_D
    \end{align}
together with the element-wise constraints $\Lambda_\rho \geqslant 0$ and $\Lambda_m = 0$ as soon as $\Lambda_\rho = 0$.
    When \(\overline{\Lambda}=(\overline{\Lambda}_\rho, \overline{\Lambda}_m,\overline{\Lambda}_\eta )\) attains the minimum, we call \(\overline{\Lambda}_\rho\) an optimal measure and \(\overline{\Lambda}_m\) an optimal momentum.
\end{definition}

\begin{remark}
Note that~\eqref{eq:continuity_equation_vanishing_viscosity} is a discrete version of the continuity equation with a viscosity term. The temporal boundary conditions are encoded in the equation as expanding~\eqref{eq:continuity_equation_vanishing_viscosity} reads
\begin{equation*}
\begin{cases}
\partial_{-,\Delta t}\Lambda_\rho - \nabla_D ^\top \Lambda_m + \varepsilon \Delta_D \Lambda_\rho = 0, \\
\displaystyle{\Lambda_\rho^{-1}  = \Pi \mu - \frac{\Delta_{-}\Lambda_\eta}{\Delta x}} , \quad \Lambda_\rho^{N_T-1} = \Pi \nu. 
\end{cases}
\end{equation*}
Here $\nabla_D^\top $ is the adjoint of the discrete gradient $\nabla_D$ with respect to the standard Euclidean product, hence $-\nabla_D^\top $ is a discrete analogue of divergence. The operator $\partial_{-,\Delta t}$ is the backward discrete time derivative given as  $\partial^i_{-,\Delta t} \Lambda_\rho := (\Lambda_\rho^{i} - \Lambda_\rho^{i-1})/\Delta t$ for $i \in \{ 0,1, \ldots, N_T -1 \}$ provided that $\Lambda^{-1}_\rho$ is defined as in the second line above.

Thus it is tempting to think of the primal problem (Definition~\ref{def:primal_problem_vanishing_viscosity}) as the discretization of 
\begin{align*}
        \inf_{\rho, m} \int^1_0 \int_\Omega L\left( \frac{\ddr m}{\ddr \rho}\right) \, \ddr \rho  \ddr t \quad  \text{such that} \quad 
        \begin{cases}
        \partial_t \rho + \nabla \cdot m = - \varepsilon \Delta \rho\\
        \rho_0 =\mu, \quad \rho_1=\nu.
        \end{cases} 
    \end{align*}
Said differently: from the beginning the parameter $\varepsilon$ was interpreted in the dual problem as a regularization parameter, as it introduces a diffusive term $\varepsilon \Delta \overline{\varphi}$ in the Hamilton-Jacobi equation. The computation we made shows that, in the primal problem, the same parameter $\varepsilon$ also adds diffusion, but this time in the continuity equation which becomes $\partial_t \rho + \nabla \cdot m = - \varepsilon \Delta \rho$. 
In particular, the case of \(L(v)=|v|^2/2\) amounts to the entropic regularized optimal transport \cite{Gentil2015AboutTA}. 

However, this analogy is not perfect. First, because of the additional variable $\Lambda_\eta$ used to constrain the discrete gradient in the dual formulation, but also because $\varepsilon$ depends on $\Delta x$ and vanishes as $\Delta x$ goes to $0$. Nevertheless, this analogy explains why our discrete solutions would be slightly more smooth than the continuous solutions (and less and less as the stepsize decreases), something that we observe numerically in Section \ref{sec:numerics}. 
\end{remark}

We conclude this section by stating that our exchange between infimum and supremum was formal, but can be made rigorous.

\begin{proposition}[Strong duality of discrete optimal transport]\label{prop:no_duality_gap_vanishing_viscosity}
    There is no duality gap between the dual problem (Definition \ref{def:dual_problem_vanishing_viscosity}) and the primal problem (Definition \ref{def:primal_problem_vanishing_viscosity}) \ie we have 
    \begin{align*}
        F_D \overline{\Phi} = L\left(\frac{\overline{\Lambda}_m}{\overline{\Lambda}_\rho}\right)\cdot \overline{\Lambda}_\rho + R \left\|\overline{\Lambda}_\eta\right\|_{L^1(\Omega_D)},
    \end{align*}
    for a maximizer \(\overline{\Phi}\) of the dual problem and a minimizer \(\overline{\Lambda}\) of the primal problem.  
    \begin{proof}
        Slater's condition states that the strong duality holds if the feasible region has non-empty interior  \cite[Sections 5.2.3 \& 5.3.2]{Boyd2006convex}. 
        In our case, it suffices to show the existence of a strictly admissible competitor of the dual problem \ie \(\Phi\in \RR^{Q_D}\) satisfying,
        \begin{equation*}
            A_t\Phi+H(A_x \Phi)< 0, \qquad |A_R \Phi| < R. 
        \end{equation*}
        It is attained for instance by defining \(\Phi^i_j\coloneqq - i\Delta t H(0) -\epsilon \) with some \(\epsilon>0\).
    \end{proof}
\end{proposition}

\section{Convergence of the optimal transport cost}
\label{sec:convergence_of_cost}

We now state and prove our main result on the quantitative convergence of the optimal transport cost. This result is valid for a general monotone and consistent scheme.

\begin{theorem}[Convergence of transport cost]\label{th:convergence_of_cost}
    Let \(\mu,\nu\in \cP(\Omega)\) and let \(\cK(\mu,\nu)\) and \(\cK_D(\mu,\nu)\) be the continuous and the discrete optimal transport cost; see respectively~\eqref{eq:dual} and Definition~\ref{def:discrete_dynamic_optimal_transport}. Then there is a constant \(C\) depending only on \(\Omega\), \(L\) and \(R\) such that
    \begin{align*}
       | \cK(\mu,\nu) - \cK_D(\mu,\nu) | \leqslant C \sqrt{h}.
    \end{align*}
\end{theorem}

To make the exposition of the proof simple, let us introduce the following notations: for a function \(\varphi : Q\to \RR\) we define the functionals $F$ and $F_D$ via
\begin{equation*}
F \varphi \coloneq
     \int_{\Omega} \varphi(1,\cdot) \, \ddr \nu  - \int_{\Omega} \varphi(0,\cdot) \,  \ddr \mu, \qquad
F_D \varphi \coloneq
    \int_{\Omega} \varphi(1,\cdot) \, \ddr \Pi\nu  - \int_{\Omega} \varphi(0,\cdot)  \, \ddr \Pi\mu.     
\end{equation*}
Our goal is to control the gap between \(F\overline{\varphi}\) and \(F_D\overline{\Phi}\) for the continuous and discrete optimal potentials  \(\overline{\varphi}\) and \(\overline{\Phi}\) respectively.
However, the relation between \(\overline{\varphi}\) and \(\overline{\Phi}\) is unclear as they are solutions to two different problems. To circumvent this issue, we replace them with functions that can be easily compared with each other.
For this purpose, we exploit not only a solution to the discrete Hamilton-Jacobi equation, but also a viscosity solution to the continuous Hamiltonian-Jacobi equation. Before proving the theorem, we review elementary results about errors caused by discretizing measures. 

\begin{lemma}[Gap between continuous and discrete measures]\label{lem:continuous discrete integral gap}
    Let \(\varphi:Q\to \RR\) be a Lipschitz function. Then we have a control, 
    \begin{align*}
        |F\varphi-F_D\varphi|\leqslant \frac{\sqrt{d}}{2} \Delta x \left( \operatorname{Lip}(\varphi(0,\cdot)) + \operatorname{Lip}(\varphi(1,\cdot)) \right).
\end{align*}
\end{lemma}
    
\begin{proof}
First note that the Wasserstein-1 distances \( W_1(\Pi\mu,\mu)\) and \( W_1(\Pi\nu,\nu)\) are bounded by \(\sqrt{d} \Delta x/2\) since each mass moves up to \(\sqrt{d} \Delta x/2\) via discretization. The claim follows from the fact that,
 for any Lipschitz function \(f:\Omega\to \RR \) and \(\rho_1, \rho_2\in \cP(\Omega)\) we have,
\begin{equation*}
    \int_\Omega f \, \ddr (\rho_1-\rho_2)\leqslant \operatorname{Lip}(f) W_1(\rho_1,\rho_2). \qedhere
\end{equation*}
\end{proof}

\begin{lemma}[Gap between discrete integral of functions]\label{lem:discrete integral bound}
    The functional \(F_D\) is \(2\)-Lipschitz with respect to \(L^\infty(Q_D)\) norm.
    \begin{proof}
        The claim follows from the definition as $\Pi \mu$, $\Pi \nu$ are probability measures. 
    \end{proof}
\end{lemma}

\begin{proof}[\textbf{Proof of Theorem \ref{th:convergence_of_cost}}]
Let \(\overline{\varphi}\) and \(\overline{\Phi}\) be respectively the solutions to the continuous and discrete problem. Note that we can take $\overline{\varphi}$ to satisfy the Lipschitz bound~\eqref{eq:bound_of_optimal_potential} thanks to Proposition~\ref{prop:bound_of_optimal_potential}. One the other hand $\overline{\Phi}$ satisfies a similar bound~\eqref{eq:discrete_initial_Lipschitz} by design. 

We first show \( F_D\overline{\Phi} \leqslant F\overline{\varphi} + C \sqrt{h}\). Let us take the solution  to the discrete initial value problem,
    \begin{align*}
\begin{cases}
\widetilde{\Phi}^{i+1}=\cS(\widetilde{\Phi}^i), \\
 \widetilde{\Phi}^{0}= \overline{\Phi}^0.
\end{cases}    
\end{align*} 
Note that $\overline{\Phi}$ and $\widetilde{\Phi}$ have the same initial data, and that \(\overline{\Phi}\leqslant \widetilde{\Phi}\) due to the inequality constraint \eqref{eq:discrete_HJ_inequality} and the monotonicity of the scheme,
hence \(F_D \overline{\Phi} \leqslant F_D \widetilde{\Phi}\). To estimate $F_D \widetilde{\Phi}$ let us consider $\widetilde{ \varphi}$ the unique viscosity solution to the the continuous initial value problem
    \begin{equation*}
    \begin{cases}
        \partial_t \widetilde{ \varphi} + H(\nabla \widetilde{ \varphi})=0, \\
         \widetilde{ \varphi} (0,\cdot) = \operatorname{LI}\left(\overline{\Phi}^0\right),
    \end{cases}
    \end{equation*}
    where  \( \operatorname{LI}\) gives the piecewise linear interpolation of a discrete function.  When \(d=1\), it is given for a discrete space function \(\Psi\) by 
    \begin{align*}
        \operatorname{LI}(\Psi)(x)
        \coloneq \frac{(j+1)\Delta x-x}{\Delta x}\Psi_j + \frac{x-j\Delta x}{\Delta x} \Psi_{j+1}, \quad x\in [j\Delta x, (j+1)\Delta x].
    \end{align*}
    For higher dimensions, it is given by bilinear interpolation, trilinear interpolation, and so on. As $\overline{\Phi}^0 \in \cC_R$ we can check that $\widetilde{\varphi}(0,\cdot)$ is $R$-Lipschitz in each coordinate and moreover $\operatorname{Lip}(\widetilde{\varphi}(0,\cdot)) \leqslant \sqrt{d} R$. Thus in particular $\operatorname{Lip}(\widetilde{\varphi}(t,\cdot)) \leqslant \sqrt{d} R$ for any $t \in [0,1]$ (see Theorem~\ref{th:CandL_viscosity_solution}). As $\widetilde{\varphi}$ is an admissible competitor, we have $F \widetilde{\varphi} \leqslant F\overline{\varphi}$. Thus we see 
\begin{equation*}
F_D\overline{\Phi} \leqslant F_D\widetilde{\Phi} \leqslant F_D\widetilde{\Phi} + F\widetilde{\varphi} - F\widetilde{\varphi}   \leqslant F\overline{\varphi} + (F_D\widetilde{\Phi} - F_D\widetilde{\varphi}) + (F_D\widetilde{\varphi} - F\widetilde{\varphi}).  
\end{equation*}
For the term $F_D\widetilde{\Phi} - F_D\widetilde{\varphi}$ we use here Lemma~\ref{lem:discrete integral bound} followed by Theorem~\ref{th:CandL_discrete_bound}, as our constraint~\eqref{eq:discrete_initial_Lipschitz} guarantees that $\widetilde{\varphi}(0,\cdot)$ is $R$-Lipschitz in each coordinate:
    \begin{align*}
        F_D\widetilde{\Phi}- F_D\widetilde{\varphi}
        \leqslant 2\|\widetilde{\Phi}-\widetilde{\varphi}\|_{L^\infty(Q_D)} \leqslant C\sqrt{h}.
    \end{align*}
    On the other hand, it follows from Lemma~\ref{lem:continuous discrete integral gap} and the regularization effect of the continuous Hamilton-Jacobi equation (Theorem \ref{th:CandL_viscosity_solution}) that
    \begin{align*}
        F_D\widetilde{\varphi}- F\widetilde{\varphi}
        \leqslant \frac{\sqrt{d}\Delta x}{2}\left(\operatorname{Lip}(\widetilde{\varphi}(1,\cdot)) + \operatorname{Lip}(\widetilde{\varphi}(0,\cdot)) \right)
        \leqslant  \sqrt{d} \Delta x  \operatorname{Lip}(\widetilde{\varphi}(0,\cdot)) \leqslant C h,
    \end{align*}
hence we obtained the claimed inequality as $C$ can be taken independent on $\mu, \nu$ thanks to~\eqref{eq:discrete_initial_Lipschitz}.

For the other inequality \( F\overline{\varphi}\leqslant F_D\overline{\Phi} + C \sqrt{h}\), we start from an optimal potential $\overline{\varphi}$ to the continuous dual transport problem which satisfies the estimate~\eqref{eq:bound_of_optimal_potential} and is a viscosity solution to the continuous Hamilton-Jacobi equation. To transform it into a discrete competitor we consider the discrete system
    \begin{equation*}\label{eq:IVP_from_continuous_optimal_potential}
        \begin{cases}
            \widetilde{\Phi}^{i+1}=\cS(\widetilde{\Phi}^i), \\
             \widetilde{\Phi}^{0}_j= \overline{\varphi}(0,j\Delta x),
        \end{cases} 
    \end{equation*}
using the point sample of the continuous potential \(\overline{\varphi}\) as initial data.
Then its solution \(\widetilde{\Phi}\) is an admissible competitor for the discrete transport problem: this is because the discrete constraint~\eqref{eq:discrete_initial_Lipschitz} is automatically satisfied thanks to~\eqref{eq:bound_of_optimal_potential}. Thus we have \(F_D\widetilde{\Phi}\leqslant F_D\overline{\Phi} \) that we use in 
\begin{equation*}
F\overline{\varphi} = F\overline{\varphi} - F_D \overline{\varphi} + F_D \overline{\varphi} + F_D \widetilde{\Phi} - F_D \widetilde{\Phi} \leqslant F_D\overline{\Phi} + (F\overline{\varphi} - F_D \overline{\varphi}) + (F_D \overline{\varphi} - F_D \widetilde{\Phi}).    
\end{equation*}
In a similar way, we have $F\overline{\varphi} - F_D \overline{\varphi} \leqslant Ch$ thanks to Lemma~\ref{lem:continuous discrete integral gap} and the bound~\eqref{eq:bound_of_optimal_potential}. On the other hand $F_D \overline{\varphi} - F_D \widetilde{\Phi} \leqslant C \sqrt{h}$: for this estimate we rely again on Lemma~\ref{lem:discrete integral bound} followed by Theorem~\ref{th:CandL_discrete_bound}. 
\end{proof}

\begin{remark}[Optimal transport on a bounded domain]
\label{rmk:optimal_transport_domain}
We built a discrete formulation on the spatial domain \(\Omega= \RR^d / (D \ZZ^d)\) to avoid boundary conditions, and we prove our convergence results in this setting. At least from a theoretical point of view, given measures $\mu, \nu$ with compact support, it is possible to embed them in $\RR^d / (D \ZZ^d)$ with $D$ large enough such that no mass crosses the ``periodic boundary'', and thus the optimal transport on the torus between $\mu$ and $\nu$ coincides with the optimal transport on $\RR^d$. However from a numerical point of view this is problematic, as most of the domain $\Omega$ would be empty, and thus most of the nodes of the discretized domain $\Omega_D$ do not record any motion of mass. This leads to a lot of wasted computational resources.  

One way to address this issue may be to use a discretization of the Hamilton-Jacobi equation on a bounded domain. We should then impose normal boundary conditions $\partial \varphi / \partial n = 0$, being $n$ the outward normal to the domain. We refer to \cite{Rouy1992Neumann} and \cite{abgrall2003numerical} for discussions of possible discretization. However, at this point the challenge would be to adapt the proof of Theorem~\ref{th:convergence_of_cost}. As we discussed in Remark~\ref{rmk:proof_regularizing_HJ}, the main problem is that the regularizing effect~\eqref{eq:regularizing_effect_HJ} is no longer available automatically in such schemes, so that even clamping the value of the gradient at the initial time is not enough. 

We leave for future work the proposal of a discretization of dynamic optimal transport via its dual formulation which would also handle the case of convex domains of $\RR^d$. 
\end{remark}

\begin{remark}
By looking into the proof of Theorem \ref{th:convergence_of_cost}, we see that a quantitative convergence of the transport cost in our discrete problem is provable if, for the chosen scheme $\cS$, a) solutions of the discrete Hamilton-Jacobi equation converge to continuous viscosity solutions at a known rate, and b) the discrete solutions have controlled Lipschitz constants at time $t= 0$ and $t=1$. 
This may leave the room for other scheme.
The challenge though would be to design an optimization method to solve the resulting discrete problem. Typically, certain types of upwind schemes define $\cS$ piecewise, making it challenging to optimize over constraints such as $\Phi^{i+1} -\cS(\Phi^i ) \leq 0$. We stick to the vanishing viscosity scheme because of the simplicity of its implementation. 
\end{remark}

\section{Convergence of the optimizers}
\label{sec:convergence_optimizers}
We can also obtain a quantitative convergence of solutions to discrete problems using the convergence of transport cost we showed in the previous section. Note that if the functionals to optimize were uniformly strictly convex, then this would be a direct consequence of the convergence of the value, that is, of the transport cost. This is not possible here: the functional in the dual problem~\eqref{eq:dual} is only linear in $\varphi$! Actually without further assumptions on $\mu, \nu$ there is not necessarily a unique solution to the continuous problem so it is very unlikely we can prove any convergence of the optimizers.

Thus in this section we impose additional assumptions that are typically satisfied when the measures $\mu, \nu$ have some smoothness. Then, by a (now classical) study of the duality gap between the primal and dual problem, it is possible to recover some convergence of the optimizers.   

\paragraph*{Assumption on the scheme}
We give our result specifically for the discrete problem with the vanishing viscosity scheme we introduced in Section \ref{sec:discrete_OT_vanishing_viscosity}. At this point we are not aware of the exact conditions for extending the results to a general scheme as it requires also the primal formulation of the discrete problem.

\paragraph*{Assumption on the Lagrangian and the Hamiltonian}
We suppose that the Lagrangian \(L\) and the Hamiltonian \(H\) satisfy for any \(v,w\in \RR^d\),
\begin{align}\label{eq:improved_Young}
    L(v)+H(w)\geqslant v\cdot w + |f_L(v)-f_H(w)|^2,
\end{align}
for some functions \(f_L,f_H : \RR^d\to \RR^d\). This is an improvement of Young's inequality \( L(v)+H(w)\geqslant v\cdot w\) and we refer to \cite{santambrogio2018regularity} for a discussion of this assumption, including the link with Bregman divergences. We refer to two  examples of such \(L\) and \(H\) from \cite[Lemma 3.3 and Lemma 3.2]{santambrogio2018regularity}. 
If \(L\) is  uniformly convex  \ie \(D^2L \geqslant \lambda I\) for some \(\lambda >0\), we have
\begin{align*}
    L(v)+H(w)\geqslant v\cdot w + \frac{\lambda}{2}|v- \nabla H(w)|^2.
\end{align*} 
On the other hand, for \(L(v)=|v|^p/p\) and  \(H(w)=|w|^q/q\) with some \(p,q> 1\) such that \(1/p+1/q=1\), we have
\begin{align*}
    L(v)+H(w)\geqslant v\cdot w + \frac{1}{2\max \{p,q\}}|v^{p/2}-w^{q/2}|^2,
\end{align*}
\label{page_vector_power}
where the power of a vector is defined as \(v^p= v|v|^{p-1}\). Note in any case, as we know that there is equality in Young's inequality when $v = \nabla H(w)$, that we must necessarily have for $w \in \RR^d$:
\begin{equation}
\label{eq:link_fL_fH}
f_H(w) = f_L(\nabla H(w)).
\end{equation} 

\paragraph*{Assumption on the optimal potential}
Finally, we assume that an optimal potential of the continuous dual problem \(\overline{\varphi}\) is of class \(C^{1,1}\) on $[0,1] \times \Omega$. That is, \(\overline{\varphi}\) is time-space differentiable and its derivatives are time-space Lipschitz. If the measures $\mu, \nu$ are in $\RR^d$, by Caffarelli's regularity theory a sufficient condition for this is that they are supported on a uniformly convex \(C^2\) domains \(X_\mu, X_\nu\) and have Lebesgue densities \(f_\mu, f_\nu\) that are H\"older continuous and bounded from below by a strictly positive constant on \(\overline{X_\mu}, \overline{X_\nu}\) respectively \cite[Theorem 4.14]{villani2003topics}. On the periodic domain \(\Omega\), it is still the case if \(X_\mu, X_\nu\) are small enough and close to each other so that the situation reduces to the case of \(\Omega=\RR^d\) and the optimal velocity of mass \(\overline{v}=\nabla H(\nabla \overline{\varphi})\) has no discontinuity on the support of the measures. Alternatively on the torus a sufficient condition for \(\overline{\varphi} \) of class \(C^{1,1}\) is \(f_\mu, f_\nu\) being nowhere vanishing and H\"older continuous \cite[Theorem 1]{Cordero1999periodicOT}, \cite[Theorem 2.2 (iii)]{AmbrosioColomboGuidoFigalli2011semigeostrophic}.

Such a condition on \(\mu,\nu\) is still weaker than the one required for similar results in the recent work \cite{NataleTodeschi2021finitevolume}, which has to assume that the optimal density $\overline{\rho}$ solving the continuous primal problem is uniformly bounded from below in the entire domain $Q$.

\paragraph{Statement of the result}
Now we state our convergence result for optimizers. We quantify the result in terms of the discrete $L^2$ norm  \(\| \cdot\|_{L^2_{\overline{\Lambda}_\rho}(Q'_D)}\) on $Q'_D$ given by
\begin{equation*}
    \| M \|^2_{L^2_{\overline{\Lambda}_\rho}(Q_D')}= \sum_{z\in Q_D'} |M_z|^2 {\overline{\Lambda}_{\rho,z}}
\end{equation*}
for \(M\in \RR^{d\times Q_D'}\): it is norm weighted by $\overline{\Lambda}_\rho$ which is solution to the discrete primal problem (see Definition~\ref{def:primal_problem_vanishing_viscosity}). Eventually, for a continuous function $\varphi$ defined on our domain $Q$, we denote by $\Pi \varphi$ the function defined on $Q_D$ which corresponds to the pointwise evaluation of the function on grid points, that is, \((\Pi \varphi)^i_j=\varphi(i\Delta t, j\Delta x)\). 

\begin{theorem}[Norm convergence of optimizers]\label{th:convergence_of_optimizer}
    Suppose that the Lagrangian \(L\) and the Hamiltonian \(H\) of the continuous problem satisfy the inequality \eqref{eq:improved_Young} with some functions \(f_L\) and \(f_H\). Let \(\overline{\varphi}\) and \(\overline{\Phi}\) be optimal potentials for the continuous and the discrete dual problems, and let  \(\overline{\Lambda}=(\overline{\Lambda}_\rho,\overline{\Lambda}_m,\overline{\Lambda}_\eta)\) be an optimizer for the discrete primal problem. Assume that \(\overline{\varphi}\in C^{1,1}(Q)\). Then we have estimate,
    \begin{align}\label{eq:estimate_for_potential}
        \left\|f_H(\nabla_D \overline{\Phi})- f_H(\nabla_D \Pi\overline{\varphi})\right\|^2_{L^2_{\overline{\Lambda}_\rho}(Q'_D)}\leqslant C\sqrt{h},
    \end{align}
    where the constant \(C\) depends only on \(\operatorname{diam}(\Omega)\), \(L\) and \(R\). If \(f_H\) is Lipschitz on bounded sets we also have an estimate
    \begin{align}\label{eq:estimate_for_velocity}
        \left\|f_L(\overline{V})- f_L(\overline{v})\right\|^2_{L^2_{\overline{\Lambda}_\rho}(Q'_D)}\leqslant C\sqrt{h},
    \end{align}
    where $\overline{v} = \nabla H(\nabla \overline{\varphi})$ is the optimal velocity for an optimal pair \((\overline{\rho},\overline{v})\) of the continuous problem and the optimal discrete velocity \(\overline{V}\coloneqq {\overline{\Lambda}_m}/{\overline{\Lambda}_\rho}\) defined on \(\operatorname{support}(\overline{\Lambda}_\rho)\).
\end{theorem}

\begin{corollary}\label{cor:convergence_of_optimizer_quadcost}
    The estimates in Theorem \ref{th:convergence_of_optimizer} hold for the Lagrangians \(L(v)=|v|^p/p\) with \(p\leqslant 2\), and in this case they read, with $1/p + 1/q = 1$,
    \begin{align*}
     \left\|
        \left(\nabla_D \overline{\Phi} \right)^{q/2}- \left(\nabla_D \Pi\overline{\varphi} \right)^{q/2}
        \right\|^2_{L^2_{\overline{\Lambda}_\rho}(Q'_D)}\leqslant C\sqrt{h}, \qquad
    \left\|
        \overline{V} ^{p/2}- \overline{v}^{p/2}
        \right\|^2_{L^2_{\overline{\Lambda}_\rho}(Q'_D)}\leqslant C\sqrt{h}.
    \end{align*}
    In particular for $p=2$, it simplifies to
     \begin{align*}
     \left\|
        \nabla_D \overline{\Phi}- \nabla_D \Pi\overline{\varphi}
        \right\|^2_{L^2_{\overline{\Lambda}_\rho}(Q'_D)}\leqslant C\sqrt{h}, \qquad
    \left\|
        \overline{V}- \overline{v}
        \right\|^2_{L^2_{\overline{\Lambda}_\rho}(Q'_D)}\leqslant C\sqrt{h}.
    \end{align*}
    \end{corollary}

\noindent Note that a quantity such as $\left\|
        \overline{V} ^{p/2}- \overline{v}^{p/2}
        \right\|^2_{L^2_{\overline{\Lambda}_\rho}(Q'_D)}$ is \emph{not} a discrete $L^p$ norm, it is a discrete $L^2$ norm between (vectorial) powers of quantities of interest.
    
    \begin{proof}
       The map \(f_H(w)=w^{q/2} / \sqrt{2 q}\) for \(q = p/(p-1) \geqslant 2\) is Lipschitz on bounded sets, while we can take $f_L(v) = v^{p/2} / \sqrt{2 q}$, and the factor $\sqrt{2q}$ can be absorbed in the constant $C$. The second claim follows from $f_L(v)=v/2$ and $f_H(w)=w/2$ for $p=2$.
    \end{proof}

To prove the theorem we follow an argument already present in the literature. It consists in quantifying the duality gap between an optimal primal solution and a (non-optimal) dual competitor: see \cite{benamou2017variational,santambrogio2018regularity} and references therein for this argument applied to get regularity estimates at the continuous level, and \cite{NataleTodeschi2021finitevolume} where this method is used to analyze discretization of dynamic optimal transport. We first prove some primal-dual relationship at the continuous level, then estimate the duality gap at the discrete level, and finally proceed with the proof. 

\begin{proposition}[Primal-dual relation at optimality in the continuous problem]\label{prop:fL_fH_continuous}
Let \(L,H\), \(f_L,f_H\) be assumed in Theorem \ref{th:convergence_of_optimizer}. Let  \((\overline{\rho},\overline{v})\) be an optimizer for the continuous primal problem and \(\overline{\varphi}\) be an optimizer in the dual problem. Then at least $\overline{\rho}$-a.e.
\begin{equation}
\label{eq:link_v_phi_at_optimality_continuous}
\overline{v} = \nabla H ( \overline{\varphi} ),
\end{equation}
and in particular combining with~\eqref{eq:link_fL_fH} we obtain
\begin{equation}
\label{eq:link_v_phi_via_fL_fH}
f_L\left(\overline{v}\right)= f_H\left(\nabla \overline{\varphi} \right). 
\end{equation}
\end{proposition} 

\noindent When $\overline{\varphi}$ is $C^{1,1}$ the relation~\eqref{eq:link_v_phi_at_optimality_continuous} is actually a way to define $\overline{v}$ in an  unambiguous manner outside of the support of $\overline{\rho}$. The analogue for the discrete problem would be \eqref{eq:link_v_phi_optimality_discrete} that we derived formally when doing the inf-sup exchange. For the proof one could use arguments in the style of \cite[Lemma 4.1]{benamou2017variational} and even quantify a duality gap but we would have to be careful between the pairing of the continuity equation in a weak sense and the Hamilton-Jacobi equation in a viscosity sense. Rather we exploit the precise structure of the optimizers that come from the static problem, at the price of more obscure proof for the non-expert reader. 

\begin{proof}
Recall that $- \overline{\varphi}(0, \cdot)$ is an optimal Kantorovich potentials in the static dual problem~\eqref{eq:static_dual_cost}; see the proof of Proposition \ref{prop:bound_of_optimal_potential}. Thus the optimal transport map $T$ between $\mu$ and $\nu$ is given by $T(x) = x + \nabla H(\nabla \varphi(0,x))$ at least $\mu$-a.e. (see e.g.~\cite[Section 1.3]{santambrogio2015optimal}). Let us write $T_t = (1-t)\mathrm{Id} + t T$ which is the optimal transport between $\mu$ and $\overline{\rho}_t$ and is invertible~\cite[Lemma 5.29]{santambrogio2015optimal}. Proposition 5.30 in~\cite{santambrogio2015optimal} yields $\overline{v}_t = (T - \mathrm{Id}) \circ T_t^{-1}$. 

On the other hand as $\overline{\varphi}$ is smooth and solves the Hamilton-Jacobi equation we can use the method of characteristics~\cite[Section 3.3]{evans10PDE}. The characteristics are given by $t \mapsto T_t(x)$, and we know that $\nabla  \overline{\varphi}$ is constant along characteristics. That gives $\nabla  \overline{\varphi}(t,T_t(x)) = \nabla \overline{\varphi}(0,x)$. Composing with $\nabla H$ and using the definition of $T$ we indeed obtain
$\nabla H(\nabla  \overline{\varphi})(t,T_t(x)) = \nabla H(\nabla  \overline{\varphi})(0,x) = T(x) - x = \overline{v}_t(T_t(x))$. Eventually we compose on the left with $T_t^{-1}$ to get the final result. 
\end{proof}

In the next lemma, we show the key estimate: the quantification of the duality gap for the discrete problem. 

\begin{lemma}[Quantification of the duality gap in the discrete problem]\label{lem:L2bound_of_disc_nabla_H}
    Let \(L,H\), \(f_L,f_H\), \(\overline{\Phi},\overline{\Lambda}=(\overline{\Lambda}_\rho,\overline{\Lambda}_m,\overline{\Lambda}_\eta), \overline{V}= \overline{\Lambda}_m /\overline{\Lambda}_\rho \) be assumed to be as in Theorem \ref{th:convergence_of_optimizer}. Then for any admissible discrete competitor \(\Phi\), we have,
    \begin{align}\label{eq:cost_gap_bound_f_H}
        \left\|f_H\left(\nabla_D\overline{\Phi}\right)- f_H(\nabla_D \Phi)\right\|^2_{L^2_{\overline{\Lambda}_\rho}(Q'_D)}
        \leqslant F_D\overline{\Phi} - F_D\Phi.
    \end{align}

\begin{proof}
        We will actually prove
        \begin{align}
            \left\|f_L\left(\overline{V}\right)- f_H(\nabla_D \Phi)\right\|^2_{L^2_{\overline{\Lambda}_\rho}(Q'_D)}
            \leqslant F_D\overline{\Phi} - F_D\Phi,
        \end{align}
        and by setting \(\Phi=\overline{\Phi}\) we obtain \(f_L\left(\overline{V}\right)= f_H\left(\nabla_D \overline{\Phi} \right)\) on \(\operatorname{support}\left(\overline{\Lambda}_\rho\right)\) (that we already derived combining \eqref{eq:link_v_phi_optimality_discrete} and \eqref{eq:link_fL_fH}), thus the stated inequality.

Thanks to the absence of duality gap at optimality (see Proposition~\ref{prop:no_duality_gap_vanishing_viscosity}), we obtain with the discrete continuity equation $F_D=A^\top\Lambda$ \eqref{eq:continuity_equation_vanishing_viscosity},
        \begin{align*}
            F_D\overline{\Phi}-F_D\Phi 
            &= L\left(\frac{\overline{\Lambda}_m}{\overline{\Lambda}_\rho}\right)\cdot \overline{\Lambda}_\rho + 
            R \left\|\overline{\Lambda}_\eta\right\|_{L^1(\Omega_D)} - F_D\Phi
            \\
            &= L\left(\frac{\overline{\Lambda}_m}{\overline{\Lambda}_\rho}\right)\cdot \overline{\Lambda}_\rho + 
            R \left\|\overline{\Lambda}_\eta\right\|_{L^1(\Omega_D)}
              -  (A^\top \overline{\Lambda}) \cdot \Phi \\
            &= L\left(\overline{V}\right)\cdot \overline{\Lambda}_\rho + R \left\|\overline{\Lambda}_\eta\right\|_{L^1(\Omega_D)} - \overline{\Lambda}_\rho \cdot (A_t \Phi) - \overline{\Lambda}_m \cdot (A_x \Phi) - \overline{\Lambda}_\eta \cdot (A_R \Phi) \\
            & \geqslant L\left(\overline{V}\right)\cdot \overline{\Lambda}_\rho  - \overline{\Lambda}_\rho \cdot (A_t \Phi) - \overline{\Lambda}_m \cdot (A_x \Phi) \\
            &= \overline{\Lambda}_\rho \cdot \left\{
                L\left(\overline{V}\right) - \overline{V}\cdot A_x \Phi - A_t \Phi
            \right\}\\
            &\geqslant \overline{\Lambda}_\rho \cdot \left\{ -A_t \Phi - H(A_x \Phi) +  \left| f_L\left( \overline{V}\right) - f_H \left( A_x \Phi\right)\right|^2 \right\}\\
            &\geqslant  \left\|f_L\left(\overline{V}\right)- f_H(\nabla_D \Phi)\right\|^2_{L^2_{\overline{\Lambda}_\rho}(Q'_D)},
        \end{align*}
where we have used in the last inequality the constraint \eqref{eq:discrete_HJ_vanishing_viscosity} coming from the admissibility of \(\Phi\).
    \end{proof}
\end{lemma}

We next quantify the violation of the discrete constraint that occurs by discretizing a continuous admissible competitor. This is where we crucially rely on the $C^{1,1}$ regularity of $\varphi$. 

\begin{lemma}[Violation of the Hamilton-Jacobi constraint due to discretization]
\label{lemma:violation_HJ}
    Let \({\varphi}\) be an admissible potential for the continuous dual problem and \(\Pi {\varphi}\) be its discretization. Assume that \({\varphi}\in C^{1,1}({Q})\). Then the maximum violation of the discrete Hamilton-Jacobi inequality by \(\Pi {\varphi}\)  is controlled as,
    \begin{align*}
        \|\left[\partial_{\Delta t} \Pi{\varphi} + H (\nabla_D \Pi{\varphi})-\varepsilon \Delta_D \Pi{\varphi}\right]^+\|_{L^\infty(Q'_D)}
        \leqslant Ch,
    \end{align*}
    with a constant \(C\) depending only on $H$ and \(\|D^2 {\varphi} \|_{L^\infty(Q)}\). Here \([\cdot]^+\) denotes the positive part and the parameter \(\varepsilon\) satisfies the monotonicity condition \eqref{eq:monotonicity_condition} for the vanishing viscosity scheme.
    \begin{proof}
        Using the admissibility of \({\varphi}\) and a standard argument by the mean value theorem,
        we have
        \begin{align*}
            & \|\left[\partial_{\Delta t} \Pi{\varphi} + H (\nabla_D \Pi{\varphi})-\varepsilon \Delta_D \Pi{\varphi}\right]^+\| \\
            & \quad \leqslant \|\partial_{\Delta t} \Pi\varphi  - \partial_t \varphi \| + \| H (\nabla_D \Pi\varphi) - H(\nabla \varphi)\|
            + \varepsilon\|\Delta_D \Pi\varphi \| + \|\left[\partial_t \varphi + H(\nabla  \varphi) \right]^+ \|\\
            & \quad \leqslant \Delta t\|\partial_{tt} \varphi\| + \left( \Delta x \operatorname{Lip}(H,B_R)+\varepsilon \right) \sum_{k,l}\|\partial_{kl} \varphi\|,
        \end{align*}
        where we denoted \(\|\cdot\|_{L^\infty(Q'_D)}\) by \(\|\cdot\|\) for simplicity. Finally,  the condition \eqref{eq:monotonicity_condition}  ensures that $\varepsilon$ scales like $\Delta x$, so that together with the assumption \(\varphi\in C^{1,1}({Q})\), we get the stated inequality.
    \end{proof}
\end{lemma}

\begin{proof}[\textbf{Proof of Theorem \ref{th:convergence_of_optimizer}}]
    Let \(\delta_{\overline{\varphi}}\coloneqq  \|\left[\partial_{\Delta t} \Pi\overline{\varphi} + H (\nabla_D \Pi\overline{\varphi})-\varepsilon \Delta_D \Pi\overline{\varphi}\right]^+\|_{L^\infty (Q_D')}\) be the maximum violation of the discrete Hamilton-Jacobi inequality by \(\Pi\overline{\varphi}\). We define a continuous potential \(\widetilde{\varphi}\) by
    \begin{align*}
        \widetilde{\varphi}(t,x)\coloneqq \overline{\varphi}(t,x) - t \delta_{\overline{\varphi}}
    \end{align*}
    so that it is an admissible competitor for the continuous problem and \(\Pi \widetilde{\varphi}\) is an admissible competitor for the discrete problem by construction. 
    This allows us to apply Lemma \ref{lem:L2bound_of_disc_nabla_H} to obtain
    \begin{multline*}
        \left\|f_H(\nabla_D\overline{\Phi})- f_H(\nabla_D \Pi \overline{\varphi})\right\|^2_{L^2_{\overline{\Lambda}_\rho}(Q'_D)}
        =\left\|f_H(\nabla_D\overline{\Phi})- f_H(\nabla_D \Pi \widetilde{\varphi})\right\|^2_{L^2_{\overline{\Lambda}_\rho}(Q'_D)} \\
        \leqslant F_D\overline{\Phi} - F_D\Pi  \widetilde{\varphi}
        =F_D\overline{\Phi}  - F_D\overline{\varphi}+\delta_{\overline{\varphi}}.
    \end{multline*}
As $F_D\overline{\varphi} \geqslant F \overline{\varphi} - Ch$ by Lemma~\ref{lem:continuous discrete integral gap}, and that $F \overline{\varphi} \geqslant F_D\overline{\Phi} - C \sqrt{h}$ thanks to Theorem~\ref{th:convergence_of_cost}, we obtain indeed  
\begin{equation*}
\left\|f_H(\nabla_D\overline{\Phi})- f_H(\nabla_D \Pi \overline{\varphi})\right\|^2_{L^2_{\overline{\Lambda}_\rho}(Q'_D)} \leqslant C h + C \sqrt{h} + \delta_{\overline{\varphi}}.
\end{equation*}    
Then we use Lemma~\ref{lemma:violation_HJ} to bound $\delta_{\overline{\varphi}}$ and get~\eqref{eq:estimate_for_potential}.

    The estimate \eqref{eq:estimate_for_velocity} is a simple consequence. We have
    \begin{align*}
        \|\nabla \overline{\varphi}-\nabla_D \Pi \overline{\varphi}\|_{L^\infty (Q_D')}\leqslant \Delta x \|D^2  \overline{\varphi}\|_{L^\infty (Q)} \leqslant C' h,
    \end{align*}
    with some constant \(C'\geqslant 0\) by a standard argument using the mean value theorem. 
    With Proposition \ref{prop:fL_fH_continuous} and the regularity assumption for \(f_H\) we have the inequality,
    \begin{align*}
        \left\|f_L\left(\overline{V}\right)- f_L\left(\overline{v}\right)\right\|_{L^2_{\overline{\Lambda}_\rho}(Q_D')}
        &= \|f_H(\nabla_D \overline{\Phi}) -f_H(\nabla \overline{\varphi})  \|_{L^2_{\overline{\Lambda}_\rho}(Q'_D)} \\
        &\leqslant   \left\|f_H(\nabla_D \overline{\Phi})- f_H(\nabla_D \Pi\overline{\varphi})\right\|_{L^2_{\overline{\Lambda}_\rho}(Q'_D)}+ C''h.
    \end{align*}
    with some \(C'' \geqslant 0\) depending on the Lipschitz constant of $f_H$ on $B_{\operatorname{Lip}(L,\operatorname{diam}(\Omega))}$.
    Applying the estimate \eqref{eq:estimate_for_potential} concludes the proof.
\end{proof} 

\begin{remark}[Comments on the rate $\sqrt{h}$]\label{rmk:CL_determines_rate}
    We see from the proofs of Theorem \ref{th:convergence_of_cost} and Theorem \ref{th:convergence_of_optimizer} that the lowest powers of convergence come from Theorem \ref{th:CandL_discrete_bound} (Theorem 1 in \cite{CrandallLions_discreteHJ}). This means that the convergence rate $\sqrt{h}$ for the solutions of the discrete Hamilton-Jacobi equations to the continuous one determines the convergence rate of both the cost and the optimizer of the optimal transport problem. In our numerical experiments in Section \ref{sec:numerics}, we found no examples where the convergence rate is exactly $\sqrt{h}$. This suggests that the convergence for any pair of (possibly very rough) input measures is better than our theoretical result $\sqrt{h}$. While the root $\sqrt{h}$ for a discrete Hamilton-Jacobi equation (Theorem~\ref{th:CandL_discrete_bound}) is sharp because viscosity solutions may develop shocks in finite time, it is known that solutions $\overline{\varphi}$ to the dual problem can only have shocks exactly at time $t=0$ or $t=1$ (see the proof of~\cite[Theorem 5.51]{villani2003topics}). It is however still unclear to us if this property can be leveraged to improve the rate of convergence. 
Some discretizations of Hamilton-Jacobi equation with higher order of convergence for smooth inputs can be found e.g. in~\cite{cheng2007discontinuous,xiong2013priori} but the convergence is in $L^2$, and not in $L^\infty$ as in Theorem~\ref{th:CandL_discrete_bound}. 
    Note, however, that a different discretization of the Hamilton-Jacobi equation with higher rates of convergence does not automatically improve our result: one would have to make sure that the steps in the proof of Theorem \ref{th:convergence_of_cost} and Theorem \ref{th:convergence_of_optimizer} carry through, and also that the constraint~\eqref{eq:discrete_HJ_inequality} is easy to handle numerically in case the scheme is different.
\end{remark}

\section{Numerical illustrations}
\label{sec:numerics}
We numerically examine our discrete optimal transport problems and convergence results. 

\subsection{Settings}
We solve the discrete problem with the vanishing viscosity scheme introduced in Section \ref{sec:discrete_OT_vanishing_viscosity}.
For the cost function,  we choose the quadratic Lagrangian  \(L(v)=|v|^2/2\) following many works in the numerical computation of optimal transport although our theory is not limited to this specific cost function.

As for the domain, we discretize the \(1+1\) dimensional time-space domain \(Q=[0,1]\times [-\frac{1}{2}, \frac{1}{2}]\)
 with identification \((t,-\frac{1}{2})\sim (t,\frac{1}{2})\). 
We make a family of discrete domains by subdividing \(Q\) into grid points  with subdivisions \(N_T\in\{16\times 2^{n} \mid n=0,1,2,3,4,5\}\) and \(N_X=N_T\). Namely the resolution \(h=\Delta t=\Delta x\) of each discrete domain is in \(\{\frac{1}{16}, \ldots, \frac{1}{16\times 2^5}\}\).

For each input probability measure, we compute its discretization given as \eqref{eq:discrete_probability_measure} that we evaluate with explicit expressions or via numerical integration.

\paragraph{Implementation of our discretization}

For each test case, we fix a pair of probability measures \(\mu\) and \(\nu\) and numerically solve the transport problem on each discrete domain. This can be performed by any standard convex optimization framework. We can obtain an optimal potential \(\overline{\Phi}\) by solving the dual problem  (Definition \ref{def:dual_problem_vanishing_viscosity}) and a minimizer \(\overline{\Lambda}=(\overline{\Lambda}_\rho, \overline{\Lambda}_m,\overline{\Lambda}_\eta )\) by solving the primal problem (Definition \ref{def:primal_problem_vanishing_viscosity}). As an example implementation, we solved the problems using the alternating direction method of multipliers (ADMM) \cite{boyd2019ADMM}. As this method directly finds a saddle point of~\eqref{eq:saddle_point_vanishing_viscosity} the output of the algorithm contains both $\overline{\Phi}$ and $\overline{\Lambda}$. We ran ADMM until the $L^2$ norms of both the primal and dual residuals became smaller than \(10^{-5}\). Such stopping criteria are discussed in literature such as \cite{Boyd2006convex}. Similarly to observations made in other discretizations of optimal transport such as \cite{Lavenant2018}, when the mesh size $h$ decreases in our discretization, the number of necessary iterations for ADMM is unaffected while the time per iteration increases, thus the total computational times increases. 
For interested readers, we share the source code in Jupyter (Python 3); see \href{https://github.com/sdsgisd/DynamicOTwithDualFormulation}{https://github.com/sdsgisd/DynamicOTwithDualFormulation}.

\paragraph{Comparison with standard finite differences}

As a point of comparison, we also solve the same test cases with the finite difference discretization proposed by Papadakis, Peyré, and Oudet~\cite{papadakis2014optimal} while it does not guarantee any convergence result. For that end, we used the Julia code originally written by the second author to handle regularized unbalanced optimal transport~\cite{baradat2021regularized}. In this code the resulting convex optimization problem is solved with proximal splitting which makes hard an exact comparison with our ADMM implementation. The Julia code is also shared in the same repository as above. 

\paragraph{Monitoring of the errors}

We consider the following gaps between continuous and discrete quantities:
\begin{align*}
    &\epsilon_\cK\coloneqq |\cK(\mu,\nu)-\cK_D(\mu,\nu)|,\\
    &\epsilon_{\varphi}\coloneqq \left\|
        \nabla_D \Pi\overline{\varphi}-  \nabla_D \overline{\Phi}
        \right\|^2_{L^2_{\overline{\Lambda}_\rho}(Q'_D)},\\
    & \epsilon_{v}\coloneqq \left\|
         \overline{v}-\overline{V}
        \right\|^2_{L^2_{\overline{\Lambda}_\rho}(Q'_D)},\\
    &\epsilon_{\rho}\coloneqq \left\|\Pi\overline{\rho}_t-\overline{\Lambda}_\rho \right\|_{L^1(Q'_D) }.
\end{align*}
Here \(\overline{\rho},\overline{v}\) are the optimal measure, velocity, and potential for the continuous problem.
 The error \(\epsilon_\cK\) is for the transport cost as in Theorem \ref{th:convergence_of_cost} and the error \( \epsilon_v\) is for optimal velocities in Corollary \ref{cor:convergence_of_optimizer_quadcost} (a special case of Theorem \ref{th:convergence_of_optimizer}). While we do not have a proven estimate between the continuous and discrete optimal measures, we also compute the error \(\epsilon_\rho\) for a reference purpose. 
 
 We plotted these errors in \autoref{fig:error_plots}, in which we also estimated the rates $\alpha$ such that $\varepsilon \sim h^\alpha$ by a standard routine with linear regression in the log domain. We omitted the plots for the error $\epsilon_\varphi$ since the explicit expression for $\overline \varphi$ is unavailable in Test case 1, and  $\epsilon_\varphi$ overlapped almost perfectly with $\epsilon_v$ in Test case 2 and 3 as \(\overline V=\nabla_D \overline{\Phi}\) and the difference between $\epsilon_v$ and $\epsilon_{\varphi}$ originates from the discrepancy between $\nabla_D \Pi\overline{\varphi}$ and $\overline{v} = \nabla \overline{\varphi}$ which is always zero except at a few exceptional grid points in these test cases.  
 
\subsection{Test cases}

We test three simple examples with different levels of regularity. In the second and the third test cases, the measures $\mu, \nu$ do not have full support, so they do not satisfy the assumptions of the previous works finding quantitative rates.  

\pgfplotsset{
  tick label style = {font=\scriptsize},
  every axis label = {font=\scriptsize},
  legend style = {font=\scriptsize, at={(0.5,-0.2)},anchor=north},
  legend columns=2,
  legend cell align=right,
  label style = {font=\scriptsize},
  legend entries={ $\frac{\ddr \mu}{\ddr x}$,$\frac{\ddr \nu}{\ddr x}$},
  xlabel={$x$},
  xlabel style={xshift=0pt, yshift=5pt},
ylabel={density},
ylabel style={xshift=0pt, yshift=-5pt},
title style={yshift=-4pt},
xtick={-0.4,-0.2,0.0,0.2,0.4}, 
  x tick label style={
    /pgf/number format/.cd,
    fixed,
    fixed zerofill,
    precision=1
  },
scale only axis=true,
width=0.25\textwidth
}

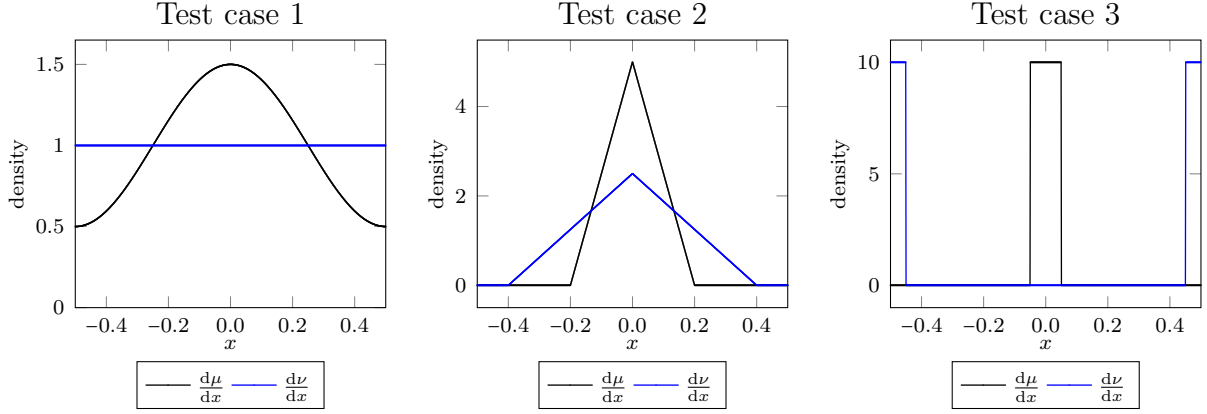
\begin{figure}[htbp]
\begin{center}
\begin{tabular}{ccc}

\begin{tikzpicture}
\begin{axis}[
title = {Test case 1},
xmin=-0.5,xmax=0.5, 
ymin=0, 
mark size=0.05pt
]
\addplot[color=black, line width = 0.5pt, mark = *] table [x=x, y=mu]{case_1_munu.txt};
\addplot[color=blue, line width = 0.5pt, mark = *] table [x=x, y=nu]{case_1_munu.txt};
\end{axis}
\end{tikzpicture}  

&

\begin{tikzpicture}
\begin{axis}[
title = {Test case 2},
xmin=-0.5,xmax=0.5,
mark size=0.05pt
]
\addplot[color=black, line width = 0.5pt, mark = *] table [x=x, y=mu]{case_2_munu.txt};
\addplot[color=blue, line width = 0.5pt, mark = *] table [x=x, y=nu]{case_2_munu.txt};
\end{axis}
\end{tikzpicture}  

& 

\begin{tikzpicture}
\begin{axis}[
title = {Test case 3},
xmin=-0.5,xmax=0.5,
mark size=0.05pt
]
\addplot[color=black, line width = 0.5pt, mark = *] table [x=x, y=mu]{case_3_munu.txt};
\addplot[color=blue, line width = 0.5pt, mark = *] table [x=x, y=nu]{case_3_munu.txt};
\end{axis}
\end{tikzpicture}  

\end{tabular}
\end{center}
\caption{Densities of \(\mu\) and \(\nu\), the initial and final measures, with respect to the Lebesgue measure. }
\label{fig:mu_and_nu}
\end{figure}

\paragraph*{Test case 1}
Our first example is a pair of smooth densities bounded from below, which is of the highest regularity among our test cases. We set
\begin{align*}
    \frac{\ddr \mu}{\ddr x}(x)=1+\frac{1}{2} \cos(2\pi w x), \qquad
    \frac{\ddr \nu}{\ddr x}(x)=1,
\end{align*}
with a parameter \(w\in \ZZ \setminus 0\), which we set \(w=1\) (\autoref{fig:mu_and_nu}).
The optimizers and the transport cost are
\begin{equation*}
\frac{\ddr \overline\rho}{\ddr x}(t,x)=\frac{1+\frac{1}{2}\cos(2w\pi T_{t}^{-1}(x))}{1+\frac{t}{2}\cos(2w\pi T_{t}^{-1}(x))}, \quad \overline{v}(t,x)=\frac{1}{4\pi w}\sin(2\pi w T_t^{-1}(x)), \quad \cK(\mu,\nu)=\frac{1}{64\pi^2 w^2}.
\end{equation*}
where the inverse of the time-$t$ optimal transport map $T_t(y)=y+t\sin(2\pi w y)/4\pi w$ can be computed by the Newton-Raphson method, for instance. Note that the expression for $\overline{\rho}$ was obtained through the change of variable formula $\ddr \overline \rho_t/\ddr x (T_t(y),t) = \ddr \mu / \ddr x(y) (T_t'(y))^{-1}$.
The input measures have absolutely continuous densities bounded from below, which is a condition required in the previous work for the quantitative convergence \cite{NataleTodeschi2021finitevolume}. We observed that all the errors are decreasing linearly or faster than the stepsize $h$ as in Figure \ref{fig:error_plots}, thus faster than our upper bound.

\paragraph*{Test case 2}
We next test the transport between triangular densities. We set \(\mu\) and \(\nu\) by,
\begin{align*}
\frac{\ddr \mu}{\ddr x}(x)=\frac{1}{w^2}\max(w-|x|,0), \qquad    \frac{\ddr \nu}{\ddr x}(x)=\frac{1}{4w^2}\max\left(2w-|x|,0\right),
\end{align*}
with a parameter \(0<w<1/4\), which we set \(w=0.2\) (\autoref{fig:mu_and_nu}). The  optimizers and the transport cost are
\begin{multline*}
    \frac{\ddr \overline{\rho}}{\ddr x}(t,x)=\frac{1}{(1+t)^2w^2}\max\left((1+t)w-|x|,0\right), \quad
    \overline{\varphi}(t,x)=\frac{x^2}{2(1+t)}, \\
    \overline{v}(t,x)=\frac{x}{1+t}, \quad
    \cK(\mu,\nu)=\frac{w^2}{12}.
\end{multline*}
The measure \(\overline{\rho}\) is expanded and flattened in time as in \autoref{fig:densities}. Even though $\mu$ and $\nu$ are not supported on the whole space, $\overline{\varphi}$ is of class $ C^{1,1}$ thus both Theorem~\ref{th:convergence_of_cost} and Theorem~\ref{th:convergence_of_optimizer} apply. All 
the errors are empirically decreasing linearly or faster as in \autoref{fig:error_plots}. 

\paragraph*{Test case 3}
Our last example is the breaking of a unimodal density toward another one, where the potential is not of class $C^{1,1}$. We set \(\mu\) and \(\nu\) as characteristic functions,
\begin{align*}
    \frac{\ddr \mu}{\ddr x}(x)=\frac{1}{2w}  \characteristic_{|x|\leqslant w}, \qquad
    \frac{\ddr \nu}{\ddr x}(x)=\frac{1}{2w} \characteristic_{1/2-|x|\leqslant w},
\end{align*}
with a parameter \(0<w<1/2\), which we set \(w=0.05\) (\autoref{fig:mu_and_nu}). The  optimizers and the transport cost are
\begin{equation*}
    \frac{\ddr \overline{\rho}}{\ddr x}(t,x)=\frac{1}{2w}\characteristic_{0\leqslant |x|-ts\leqslant w}, \quad 
    \overline{\varphi}(t,x)=|x|s-\frac{s^2t}{2}, \quad 
    \overline{v}(t,x)=s \operatorname{sign}(x),\quad
    \cK(\mu,\nu)=\frac{s^2}{2}
\end{equation*}
where all the mass travel the same distance \(s=\frac{1}{2}-w\).
Note that the mass travels in two directions as seen in \autoref{fig:densities} and the velocity field \(\overline{v}\) is discontinuous at \(x=0\). 
We are in the situation where Theorem~\ref{th:convergence_of_cost} applies, but the assumptions of Theorem~\ref{th:convergence_of_optimizer} are not satisfied. Though our rates are again not sharp, the convergence of the transport cost seems to be a bit worse than in the first case, between \(\sqrt{h}\) and \(h\) as can be seen in \autoref{fig:error_plots}.

\begin{figure}
\centering    
\hrule 
\vspace{5pt}
Test case 1 \\ \vspace{5pt}

    \begin{minipage}[t]{0.3\textwidth}
        \includegraphics[height=1.8in]{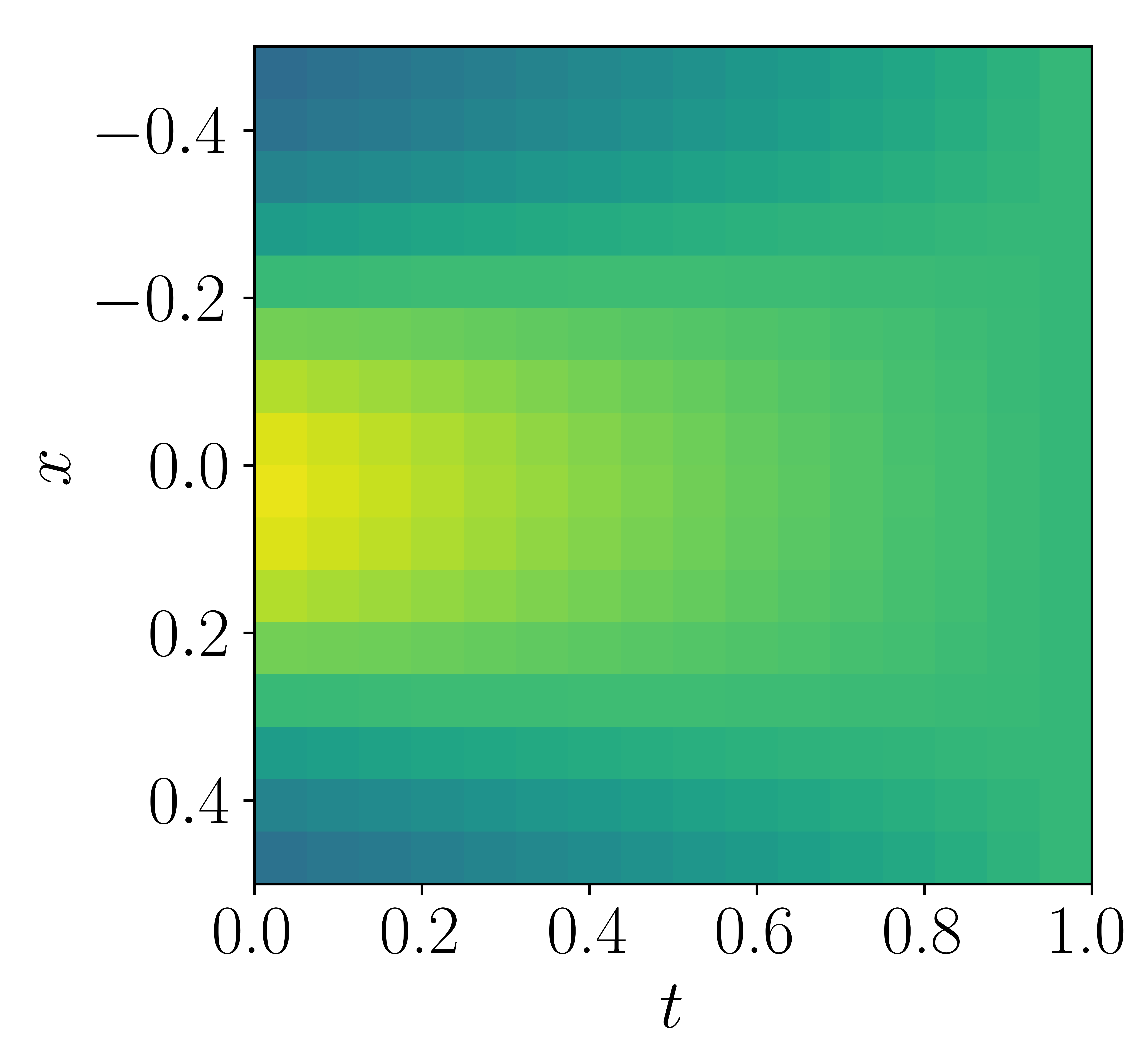}
        \end{minipage}
         \begin{minipage}[t]{0.3\textwidth}
        \includegraphics[height=1.8in,angle=0]{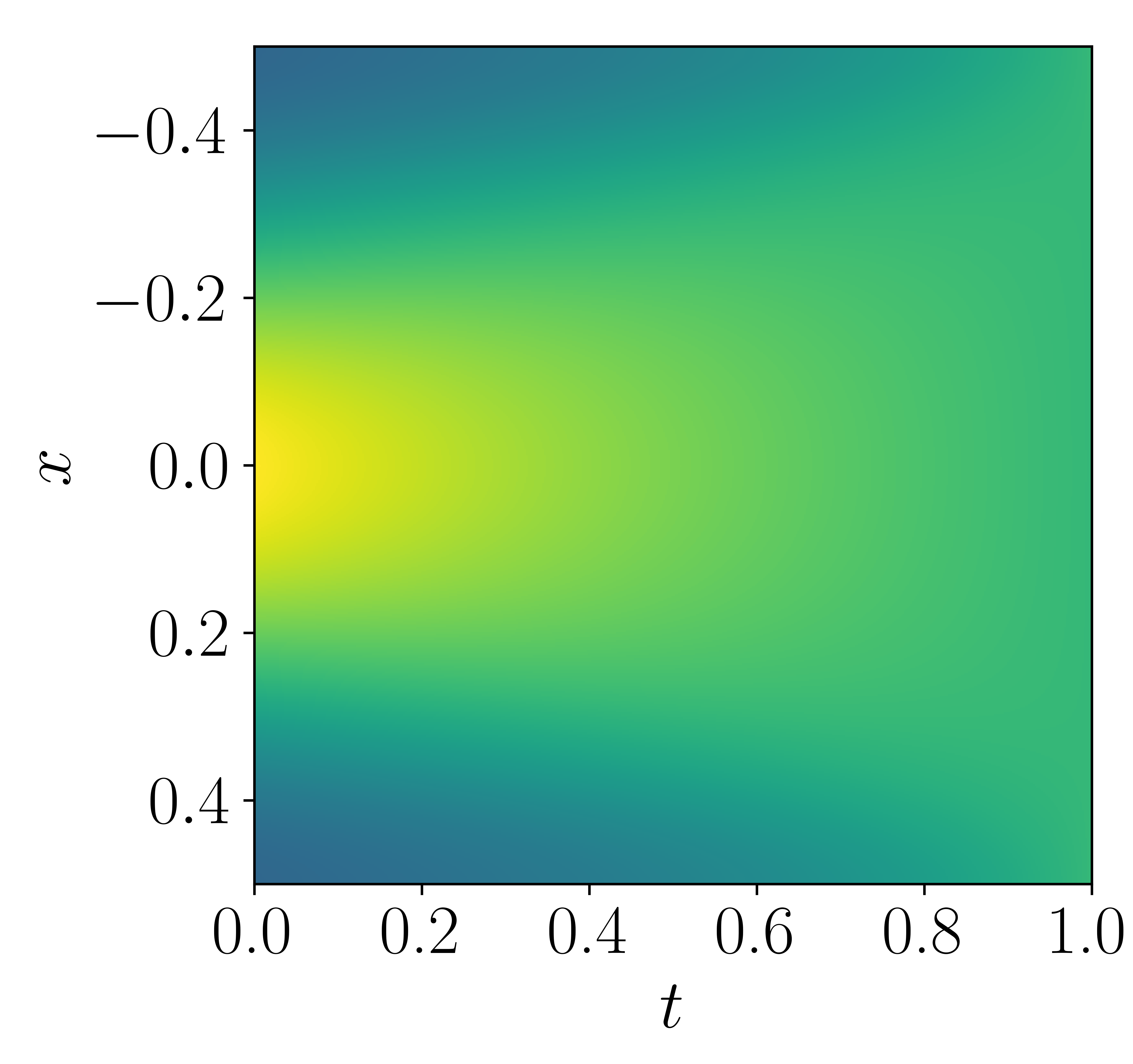}
        \end{minipage}
         \begin{minipage}[t]{0.3\textwidth}
        \includegraphics[height=1.8in,angle=0]{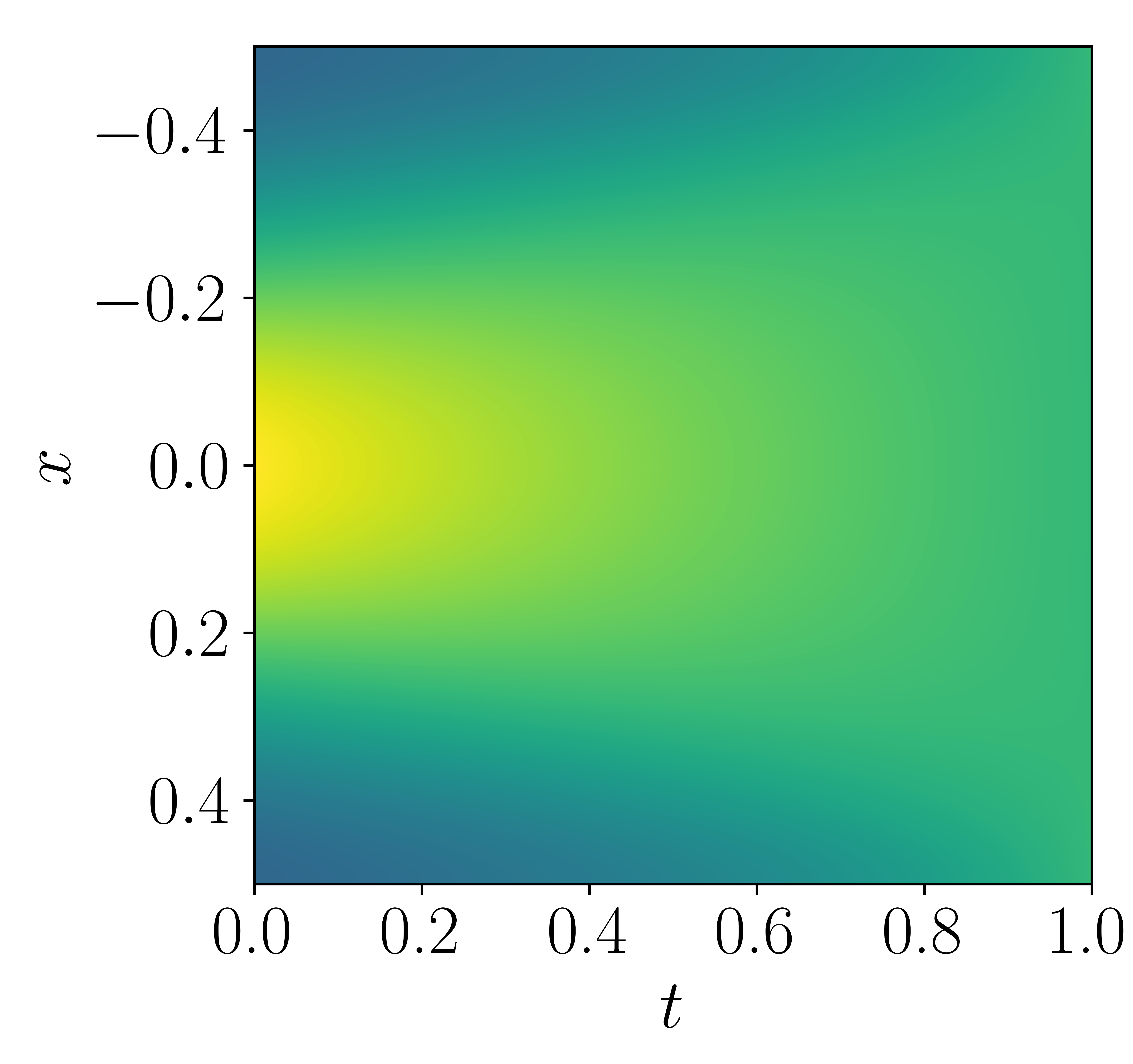}
        \end{minipage}
        \begin{minipage}[t]{0.05\textwidth}
        \includegraphics[height=1.8in,angle=0]{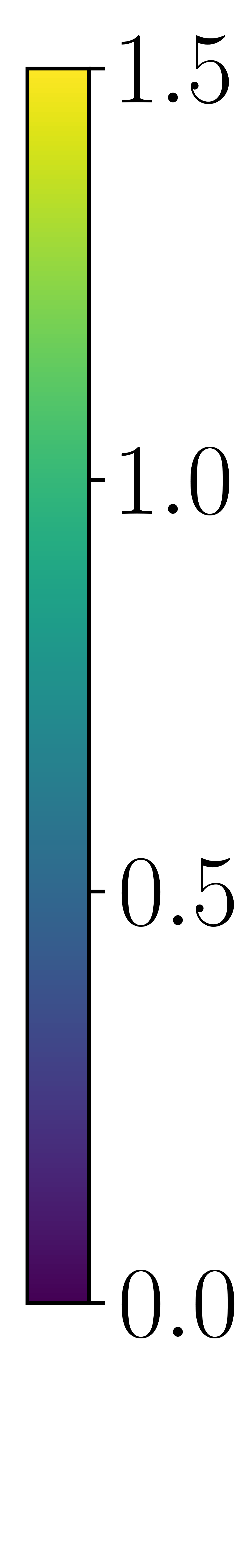}
        \end{minipage}
        
\hrule 
\vspace{5pt}

Test case 2\\ \vspace{5pt}
    
    \begin{minipage}[t]{0.3\textwidth}
        \includegraphics[height=1.8in]{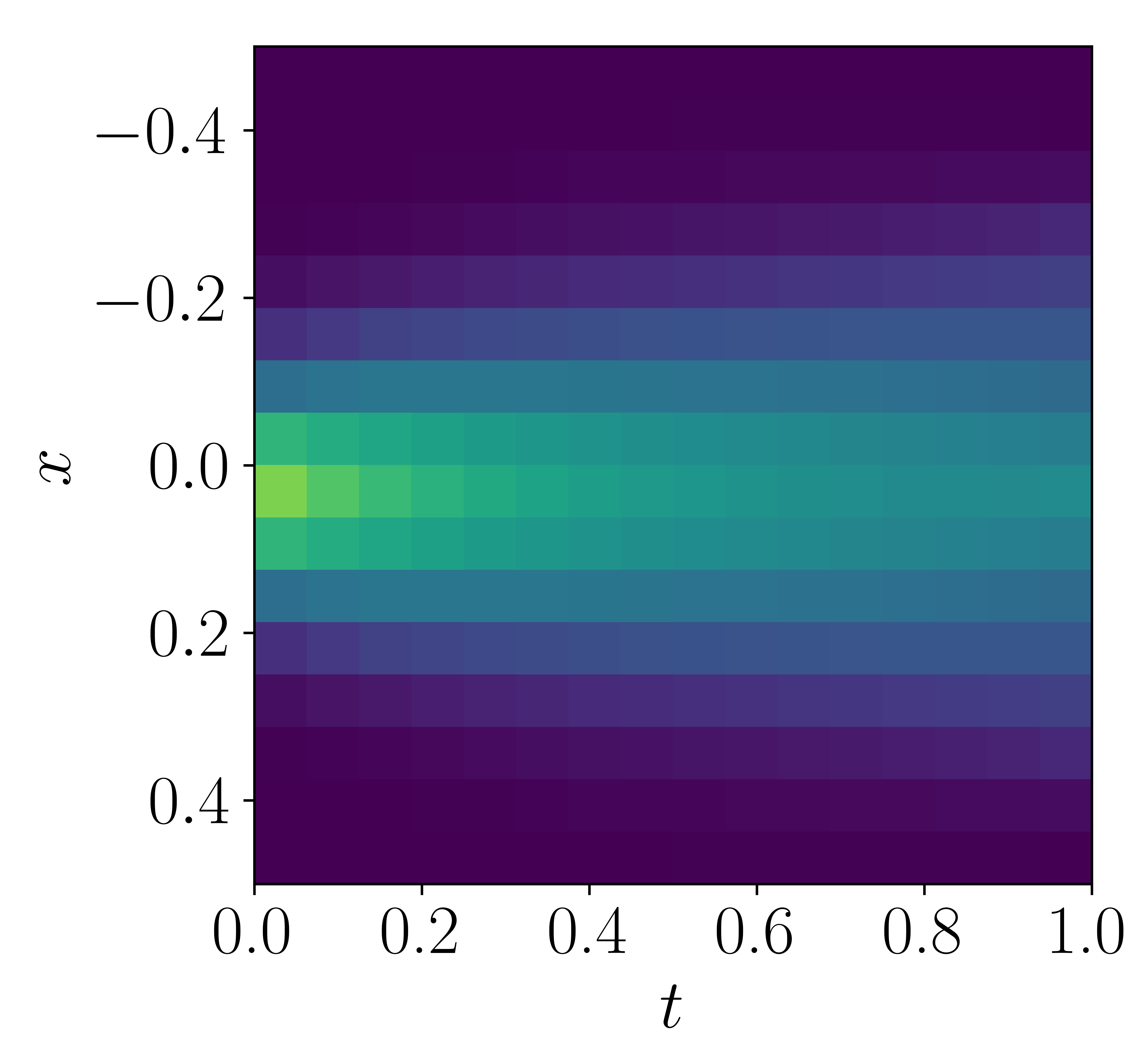}
        \end{minipage}
         \begin{minipage}[t]{0.3\textwidth}
        \includegraphics[height=1.8in,angle=0]{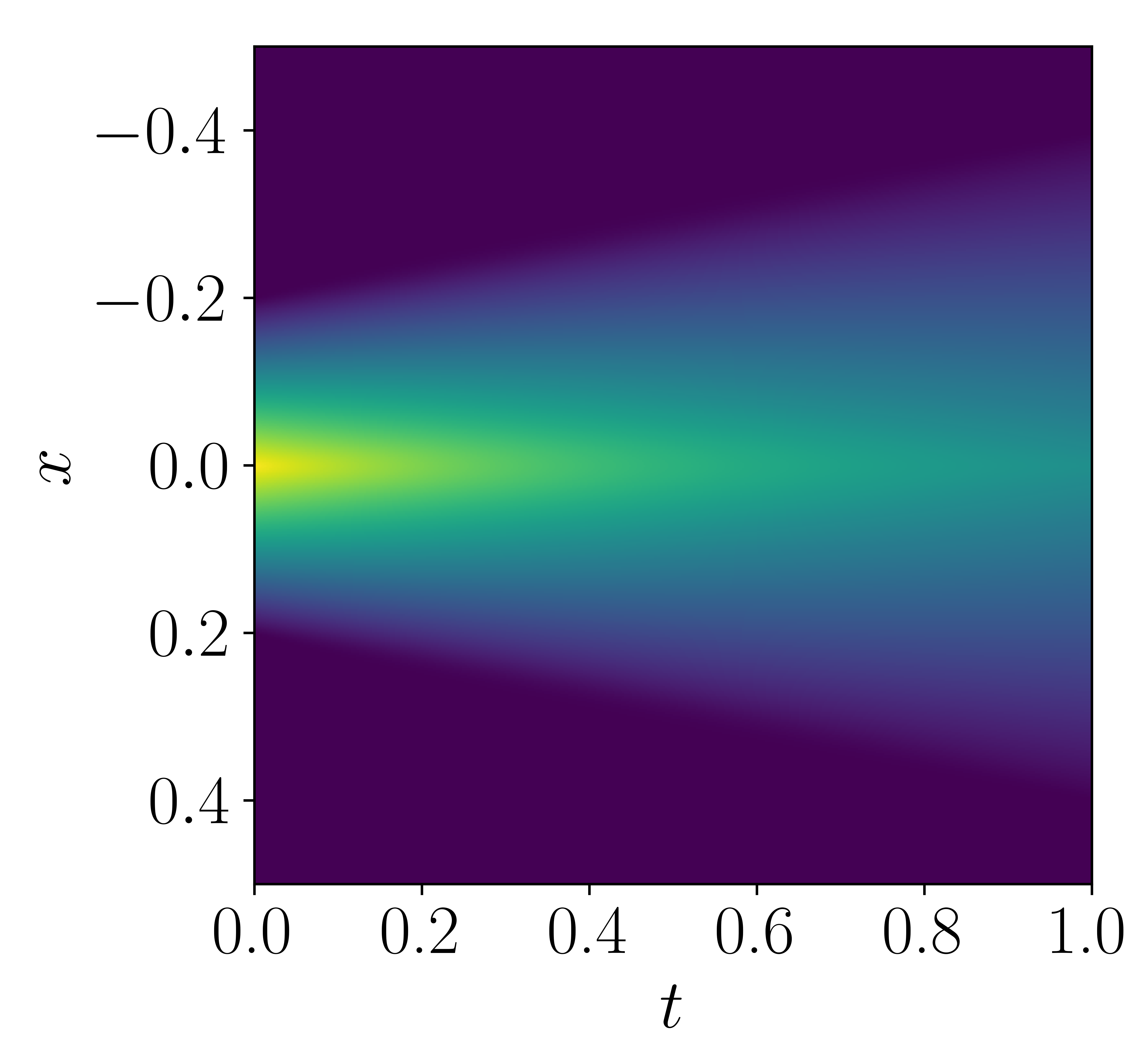}
        \end{minipage}
         \begin{minipage}[t]{0.3\textwidth}
        \includegraphics[height=1.8in,angle=0]{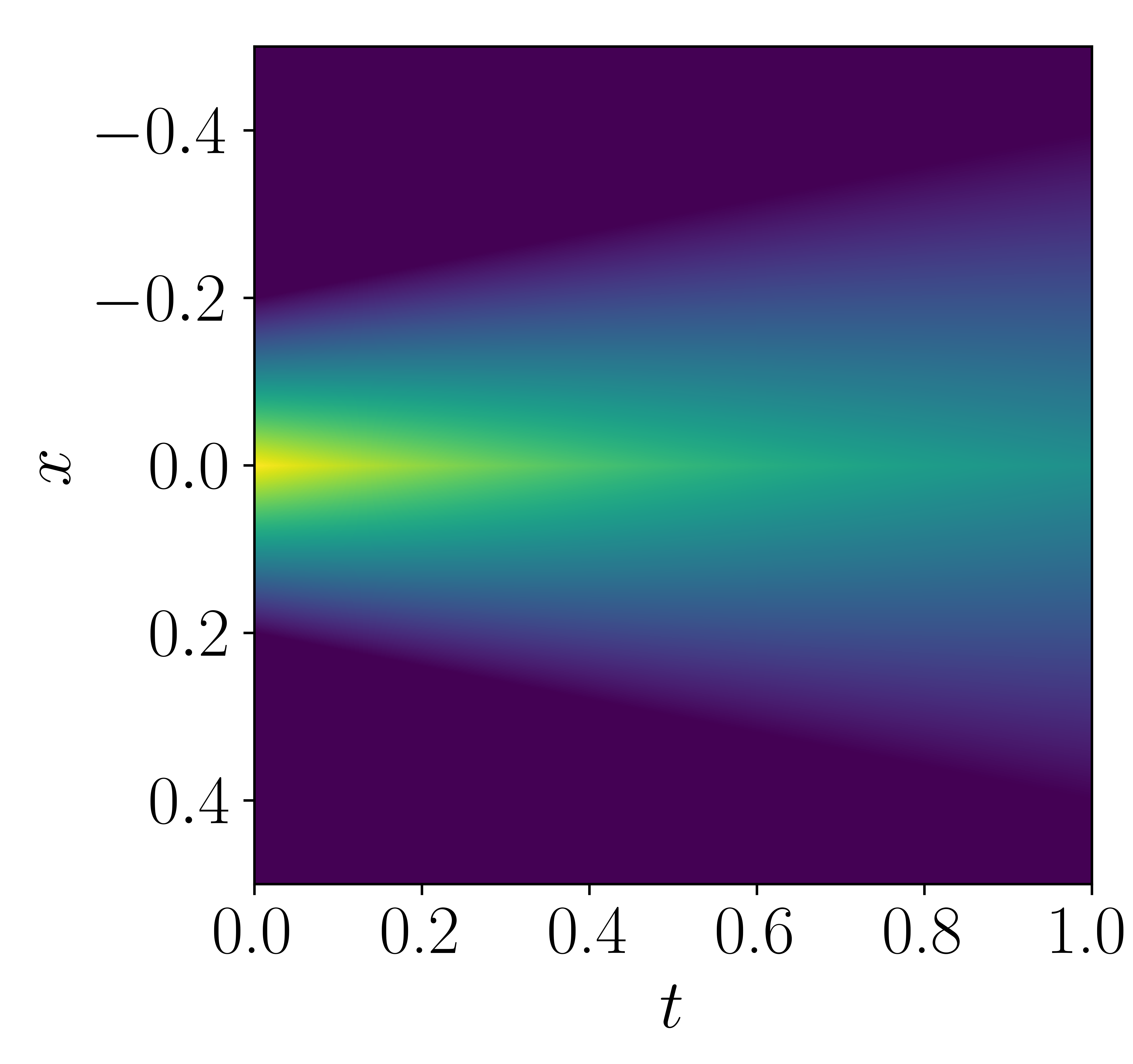}
        \end{minipage}
        \begin{minipage}[t]{0.05\textwidth}
        \includegraphics[height=1.8in,angle=0]{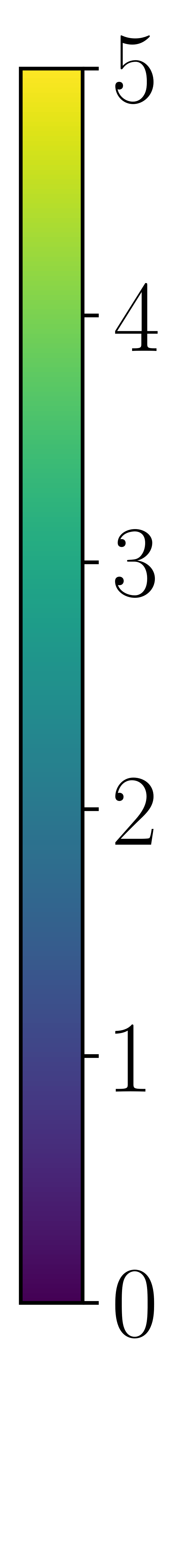}
        \end{minipage}
        
\hrule 
\vspace{5pt}

Test case 3 \\ \vspace{5pt}

    \begin{minipage}[t]{0.3\textwidth}
        \includegraphics[height=1.8in]{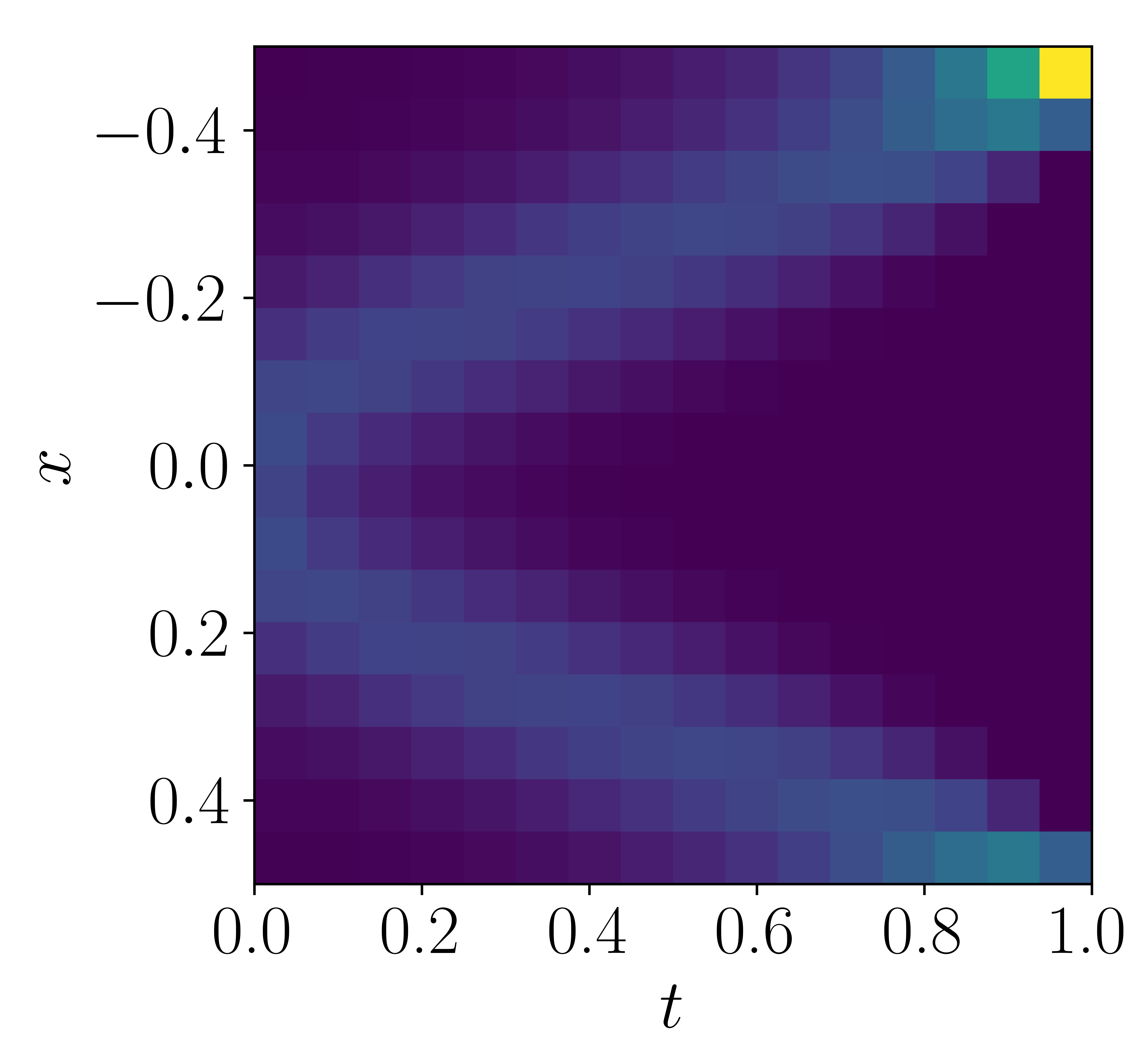}
        \end{minipage}
         \begin{minipage}[t]{0.3\textwidth}
        \includegraphics[height=1.8in,angle=0]{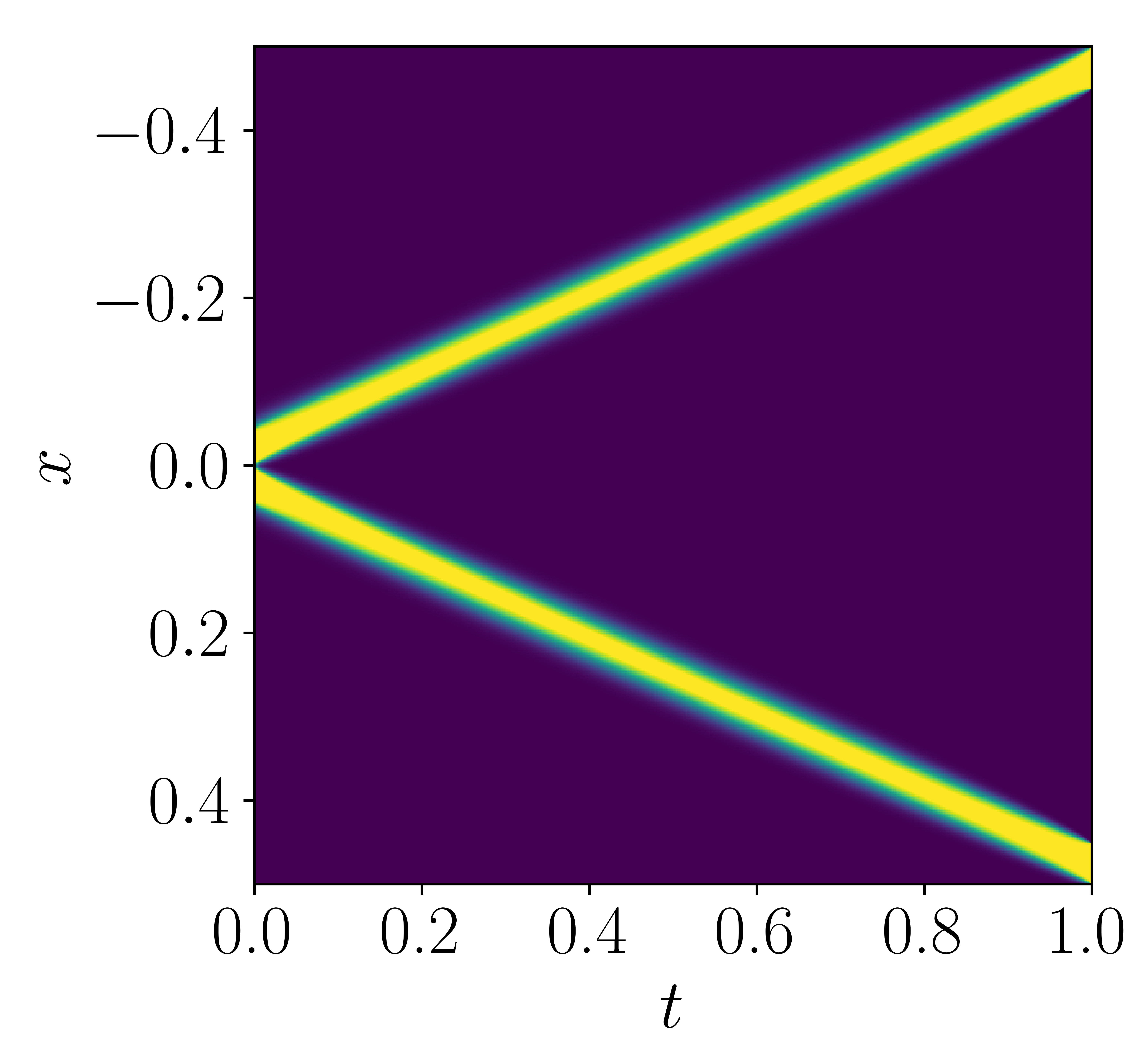}
        \end{minipage}
         \begin{minipage}[t]{0.3\textwidth}
        \includegraphics[height=1.8in,angle=0]{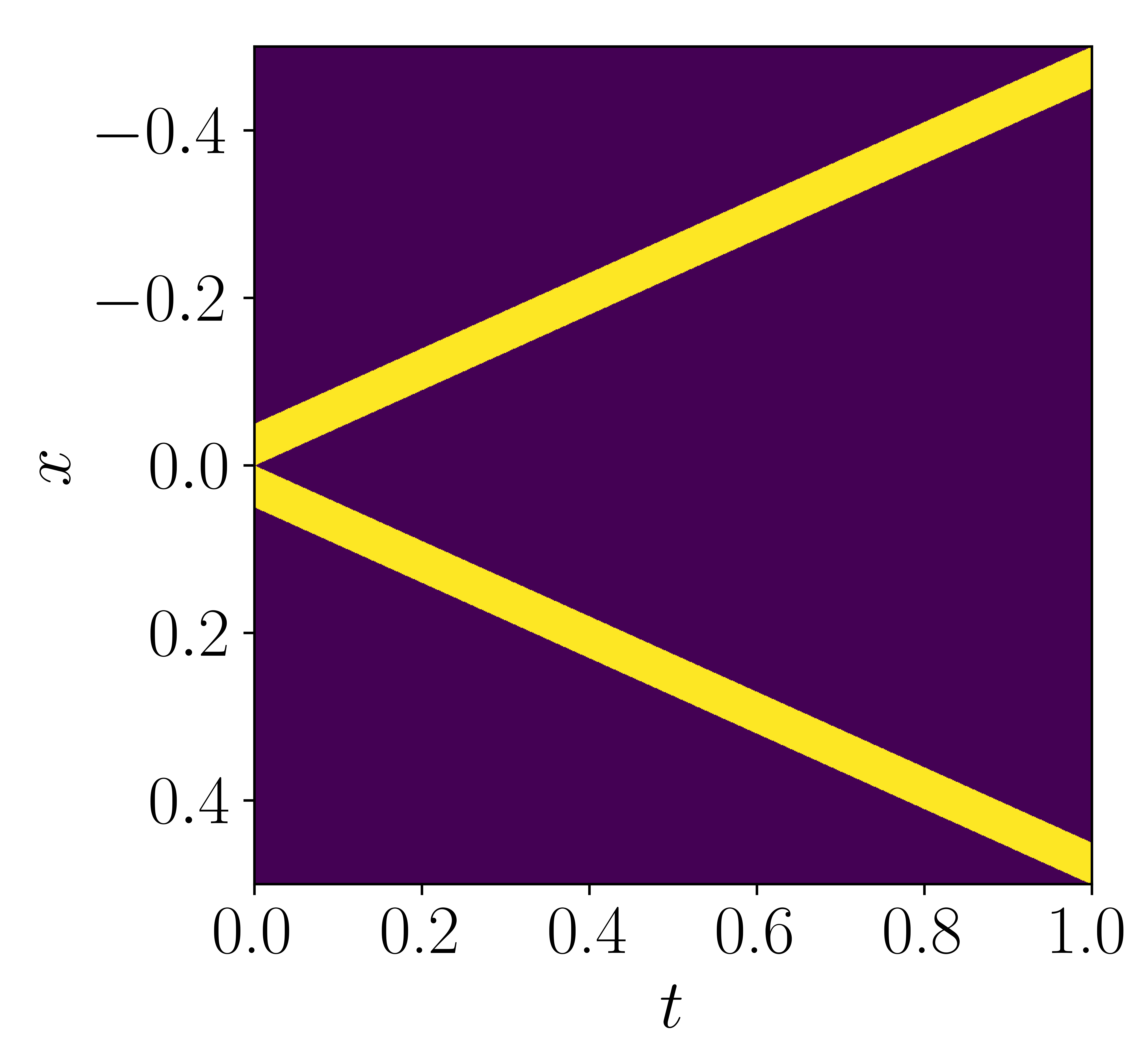}
        \end{minipage}
        \begin{minipage}[t]{0.05\textwidth}
        \includegraphics[height=1.8in,angle=0]{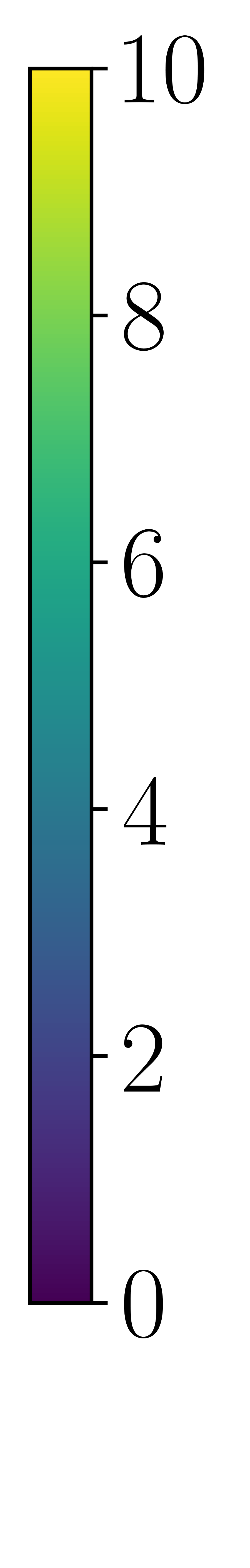}
        \end{minipage}

    \caption{Discrete and continuous optimal measures. From left to right: the discrete optimal  measures \(\overline{\Lambda}_\rho\) for \(h=1/16\) and \(h=1/16\times 2^5\), and the continuous optimal  measure \(\overline{\rho}\).}
    \label{fig:densities}
\end{figure}

\pgfplotsset{
  tick label style = {font=\scriptsize},
  every axis label = {font=\scriptsize},
  legend style = {font=\scriptsize, at={(0.5,-0.1)},anchor=north},
  legend columns=1,
  legend cell align=right,
  label style = {font=\scriptsize},
  legend entries={ $\frac{\ddr \mu}{\ddr x}$,$\frac{\ddr \nu}{\ddr x}$},
xlabel={$h$},
xlabel style={xshift=0pt, yshift=-1pt},
ylabel={error},
title style={yshift=-4pt},
scale only axis=true,
width=0.23\textwidth,
}

\begin{figure}[htbp]
\centering 
\begin{tabular}{ccc}
\hline 
& & \\

 & Test case 1 & \\
\begin{tikzpicture}
\begin{loglogaxis}[
title = {$\varepsilon_{\mathcal{K}}$},
legend entries={PPO $\alpha_\cK=1.998$,HJ  $\alpha_\cK= 1.053$},
mark size=1.pt
]
\addplot[color=black, line width = 0.5pt, mark = *] table [x=h, y=errorValue]{PPO_case_1.txt};
\addplot[color=blue, line width = 0.5pt, mark = *] table [x=h, y=errorValue]{HJ_case_1.txt};
\end{loglogaxis}
\end{tikzpicture}  

& 

\begin{tikzpicture}
\begin{loglogaxis}[
title = {$\varepsilon_v$},
legend entries={PPO $\alpha_v=1.997$,HJ  $\alpha_v=2.027$},
mark size=1.pt
]
\addplot[color=black, line width = 0.5pt, mark = *] table [x=h, y=errorV]{PPO_case_1.txt};
\addplot[color=blue, line width = 0.5pt, mark = *] table [x=h, y=errorV]{HJ_case_1.txt};
\end{loglogaxis}
\end{tikzpicture}  

&

\begin{tikzpicture}
\begin{loglogaxis}[
title = {$\varepsilon_\rho$},
legend entries={PPO $\alpha_\rho=2.254$,HJ  $\alpha_\rho=1.070$},
mark size=1.pt
]
\addplot[color=black, line width = 0.5pt, mark = *] table [x=h, y=errorRho]{PPO_case_1.txt};
\addplot[color=blue, line width = 0.5pt, mark = *] table [x=h, y=errorRho]{HJ_case_1.txt};
\end{loglogaxis}
\end{tikzpicture} 
 \\
 & & \\
\hline 
& & \\
 &  Test case 2 & \\
\begin{tikzpicture}
\begin{loglogaxis}[
title = {$\varepsilon_\cK$},
legend entries={PPO $\alpha_\cK=1.873$,HJ  $\alpha_\cK=1.128$},
         mark size=1.pt
]
\addplot[color=black, line width = 0.5pt, mark = *] table [x=h, y=errorValue]
{PPO_case_2.txt};
\addplot[color=blue, line width = 0.5pt, mark = *] table [x=h, y=errorValue]{HJ_case_2.txt};
\end{loglogaxis}
\end{tikzpicture}  

& 

\begin{tikzpicture}
\begin{loglogaxis}[
title = {$\varepsilon_v$},
legend entries={PPO $\alpha_v=2.015$,HJ  $\alpha_v=1.772$},
mark size=1.pt
]
\addplot[color=black, line width = 0.5pt, mark = *] table [x=h, y=errorV]
{PPO_case_2.txt};
\addplot[color=blue, line width = 0.5pt, mark = *] table [x=h, y=errorV]{HJ_case_2.txt};
\end{loglogaxis}
\end{tikzpicture}  

&

\begin{tikzpicture}
\begin{loglogaxis}[
title = {$\varepsilon_\rho$},
legend entries={PPO $\alpha_\rho=1.228$, HJ  $\alpha_\rho=0.878$},
mark size=1.pt
]
\addplot[color=black, line width = 0.5pt, mark = *] table [x=h, y=errorRho]
{PPO_case_2.txt};
\addplot[color=blue, line width = 0.5pt, mark = *] table [x=h, y=errorRho]{HJ_case_2.txt};
\end{loglogaxis}
\end{tikzpicture}  
 \\
 & & \\
\hline 
& & \\
 & Test case 3 & \\
\begin{tikzpicture}
\begin{loglogaxis}[
title = {$\varepsilon_{\mathcal{K}}$},
legend entries={PPO $\alpha_\cK=1.379$,HJ  $\alpha_\cK=0.887$},
mark size=1.pt
]
\addplot[color=black, line width = 0.5pt, mark = *] table [x=h, y=errorValue]{PPO_case_3.txt};
\addplot[color=blue, line width = 0.5pt, mark = *] table [x=h, y=errorValue]{HJ_case_3.txt};
\end{loglogaxis}
\end{tikzpicture}  

& 

\begin{tikzpicture}
\begin{loglogaxis}[
title = {$\varepsilon_v$},
legend entries={PPO $\alpha_v=1.938$,HJ  $\alpha_v=0.996$},
mark size=1.pt
]
\addplot[color=black, line width = 0.5pt, mark = *] table [x=h, y=errorV]{PPO_case_3.txt};
\addplot[color=blue, line width = 0.5pt, mark = *] table [x=h, y=errorV]{HJ_case_3.txt};
\end{loglogaxis}
\end{tikzpicture}  

&

\begin{tikzpicture}
\begin{loglogaxis}[
title = {$\varepsilon_\rho$},
legend entries={PPO $\alpha_\rho=0.418$,HJ  $\alpha_\rho=0.466$},
mark size=1.pt
]
\addplot[color=black, line width = 0.5pt, mark = *] table [x=h, y=errorRho]{PPO_case_3.txt};
\addplot[color=blue, line width = 0.5pt, mark = *] table [x=h, y=errorRho]{HJ_case_3.txt};
\end{loglogaxis}
\end{tikzpicture}

\end{tabular}

\caption{Plots of the grid resolutions and the errors in the log-log domain. PPO denotes the discretization by \cite{papadakis2014optimal} and HJ denotes our proposed one. The $\alpha$ values are the approximate rates of convergence.
}
\label{fig:error_plots}
\end{figure}
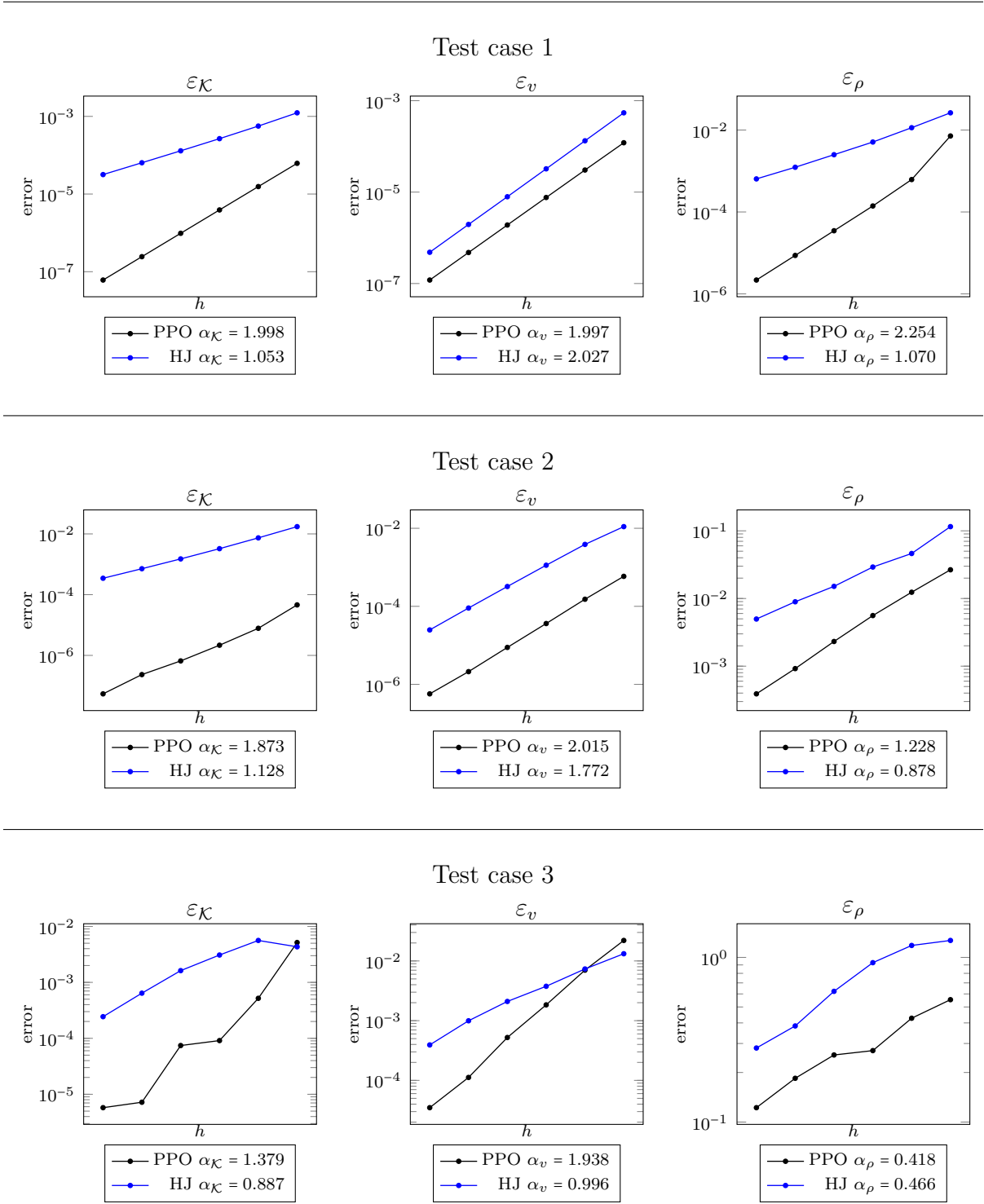

\begin{remark}
As we discussed in Section \ref{sec:discrete_OT_vanishing_viscosity}, the viscosity coefficient \(\varepsilon\) vanishes as the resolution \(h\to 0\). 
In the both examples, we see that for moderately small values of $h$ (equivalently $\varepsilon$) there is a numerical smoothing whose effect becomes weaker as \(h\) decreases as in \autoref{fig:densities}. 
\end{remark}
\begin{remark}
    In both our test cases, we observed that our constraint~\eqref{eq:discrete_initial_Lipschitz}, which here reads \(| \Delta_+ \Phi^0/{\Delta x}| \leqslant R\coloneqq \operatorname{Lip}(L,B_{\operatorname{diam}(\Omega)})=\operatorname{diam}(\Omega)=1/2\), was effective 
    \ie \(\|\overline{\Lambda}_\eta\|_{L^1(\Omega_D)}>0\). That is, the discrete dual optimal transport problem has no regularizing effect on the potential and we really need to encode it as a constraint~\eqref{eq:discrete_initial_Lipschitz}. We can also run numerical computations with a larger $R$ (including $\infty$) for which the constraint may become ineffective. As mentioned in Remark \ref{rmk:largerR}, the scheme may not be monotone when $R$ is too large. In that case, we lose the theoretical guarantee for convergence as the control between a continuous solution and a discrete solution to the Hamilton-Jacobi equation (Theorem \ref{th:CandL_discrete_bound}) is missing. We, however, have observed no cases so far without convergence of the cost or the optimizers while we lose control of their Lipschitz constants which may be quite large though seems independent of the stepsize. Said otherwise, the method behaves well even when the theory breaks, which happens quite often in the numerical study of dynamic optimal transport.   
\end{remark}

\begin{remark}
The finite difference discretization by Papadakis, Peyré, and Oudet~\cite{papadakis2014optimal} better performed in all the examples. We speculate that this is because a) this method uses staggered grids \ie vector values and scalar values are stored in different locations, which more accurately discretizes vector quantities than the collocated grids that our discretization relies on, b) the discretization in~\cite{papadakis2014optimal} is symmetric in both time and space while ours is asymmetric in time as mentioned in Remark \ref{rmk:asymmetry_mu_nu}, and c) ours regularizes the solution of the Hamilton-Jacobi equation with viscosity.

Note, however, that convergence of the cost or the optimizer in~\cite{papadakis2014optimal} is not guaranteed. Indeed, though it would be possible to write the dual problem to their discretization and it looks like a discretization of the Hamilton-Jacobi equation, it is one for which no convergence is guaranteed  due to the lack of properties such as monotonicity.
\end{remark}

\section*{Acknowledgments}

The authors would like to thank Chris Wojtan for his continuous support and several interesting discussions. Part of this research was performed during two visits: one of SI to the BIDSA research center at Bocconi University, and one of HL to the Institute of Science and Technology Austria. Both host institutions are warmly acknowledged for the hospitality. HL is partially supported by the MUR-Prin 2022-202244A7YL ``Gradient Flows and Non-Smooth Geometric Structures with Applications to Optimization and Machine Learning'', funded by the European Union - Next Generation EU. SI is supported in part by ERC Consolidator Grant 101045083 ``CoDiNA'' funded by the European Research Council.

\bibliographystyle{plain}
\bibliography{reference}

\end{document}